\documentclass[a4paper,12pt]{amsart}
\usepackage{microtype,mathtools,a4wide, color} 

\usepackage[utf8]{inputenc}
\usepackage[alphabetic,backrefs]{amsrefs}

\renewcommand{\Re}{\operatorname{Re}}

\newcommand{\C}{\mathbb{C}}

\newcommand{\Tr}{\operatorname{Tr}}
\newcommand{\im}{\mathrm{i}}

\usepackage[hidelinks]{hyperref} 
\usepackage{xcolor}
\hypersetup{colorlinks=true,
    linkcolor=blue,   
    urlcolor=blue,     
    citecolor=green, 
}
\usepackage{relsize}

\newtheorem{thm}{Theorem}[section]
\newtheorem{rmk}[thm]{Remark}
\newtheorem{lemma}[thm]{Lemma}
\newtheorem{prop}[thm]{Proposition}
\newtheorem{cor}[thm]{Corollary}

\newtheorem{defi}[thm]{Definition}

\begin{document}
\author{Mar\'ia \'Angeles Garc\'ia-Ferrero}
\address{Instituto de Ciencias Matem\'aticas CSIC-UAM-UC3M-UCM, c/ Nicol\'as Cabrera 13--15,
28049 Madrid, Spain}
\email{\href{mailto:garciaferrero@icmat.es}{\texttt{garciaferrero@icmat.es}}}
\author{Joaquim Ortega-Cerd\`a}
\address{Dept.\ Matem\`atica i Inform\`atica,
 Universitat  de Barcelona,
Gran Via 585, 08007 Bar\-ce\-lo\-na, Spain and
CRM, Centre de Recerca Matem\`atica, Campus de Bellaterra Edifici C, 08193 
Bellaterra, Barcelona, Spain}
\email{\href{mailto:jortega@ub.edu}{\texttt{jortega@ub.edu}}}

\title[Stability of the Wehrl entropy in homogeneous polynomials]{Stability of the generalized Wehrl entropy and the local concentration of homogeneous polynomials}

\begin{abstract}
We study two notions of concentration for homogeneous polynomials of degree $N$ in $d+1$ complex variables on the unit sphere: a local notion measuring the fraction of the $L^2$-norm supported on a measurable subset; and  a global notion given by the generalized Wehrl entropy. In both cases, the extremizers are known to be reproducing kernels, that is, monomials up to a unitary rotation, by results of Lieb--Solovej. We establish stability results for both inequalities in higher dimensions. For the local concentration, we show that for sets of sufficiently small measure, the almost-maximizers are quantitatively close to reproducing kernels, both in the polynomial and in the domain, extending previous resuls in one dimension. For the generalized Wehrl entropy, we prove that for any non-linear convex function $\Phi$ and all sufficiently large degree $N$, the reproducing kernels are the unique minimizers up to stability, complementing recent results by Nicola--Riccardi--Tilli, which require a non-linearity condition near $1$ that exclude key examples such as the concentration functional. As a consequence, by passing to the large-$N$ limit, we recover stability results for both problems in the Bargmann--Fock space.
\end{abstract}

\maketitle
\tableofcontents

\section{Introduction}

Let $\mathbb P_N^d$ denote the space of homogeneous polynomials of degree $N$ in $d+1$ complex variables, in which any element can be represented by
$$Q(\zeta)=\sum_{|\alpha|=N} a_\alpha \zeta^\alpha, \quad \mbox{where } \zeta\in\C^{d+1}, \ \alpha\in\mathbb N_0^{d+1} \mbox{ and } \ a_\alpha\in\C.$$
We endow $\mathbb P_N^d$ with the inner product 
\begin{equation}\label{eq:innerprod}
    \langle P, Q\rangle_{\mathbb P_N^d}=\binom{N+d}{N}\int_{\partial\mathbb B} P(\zeta)\overline{Q(\zeta)}d\sigma(\zeta), \quad P,Q \in \mathbb P_N^d,
\end{equation}
where $d\sigma$ is the rotation-invariant  measure on the $(2d+1)$-dimensional  sphere $\partial \mathbb B$ with $\sigma(\partial \mathbb B)=1$. 
Hence, the corresponding norm for $Q\in\mathbb P^d_N$ is defined as
\begin{equation}\label{eq:norm}
    \|Q\|_{\mathbb P_N^d}^2=\binom{N+d}{N}\int_{\partial\mathbb B} |Q(\zeta)|^2d\sigma(\zeta).
\end{equation}
The prefactor $\binom{N+d}{N}$ is chosen so that $\|\zeta_k^N\|_{\mathbb P_N^d}=1$, $k\in\{1,\dots, d+1\}$ (see Lemma~\ref{lem:onb}).

The space $\mathbb P_N^d$ endowed with the  inner product \eqref{eq:innerprod} is a reproducing kernel Hilbert space with kernel
\begin{equation*}
    K_{N}(\zeta,\eta)=\sum_{|\alpha|=N} \frac{N!}{\alpha!}\zeta^\alpha\bar\eta^\alpha=(\zeta\cdot\bar\eta)^N, \quad \zeta, \eta\in\partial \mathbb B.
\end{equation*}
Then, for all $Q\in\mathbb P^d_N$ and $\eta\in\partial\mathbb B$
\begin{align*}
    Q(\eta)=\langle Q, K_N(\cdot, \eta)\rangle_{\mathbb P_N^d}=\binom{N+d}{N}\int_{\partial\mathbb B} Q(\zeta)(\bar \zeta\cdot\eta)^Nd\sigma(\zeta).
\end{align*}
We notice that the reproducing kernel for fixed $\eta\in\partial\mathbb B$ has unit norm, i.e. $\|K_N(\cdot,\eta)\|_{\mathbb P^d_N}^2 =K_N(\eta,\eta) =1$, and thus, if $Q\in \mathbb P_N^d$, with $\|Q\|_{\mathbb P^d_N} \le 1$, then 
\begin{align}\label{eq:supQ}
    \sup_{\zeta\in\partial \mathbb B} |Q(\zeta)| = |Q(\eta)| = \left|\left\langle Q, K_N(\cdot, \eta)\right\rangle_{\mathbb P_N^d}\right| \le 1.
\end{align}

This paper is devoted to the study of two notions for concentration of polynomials in $\mathbb P^d_N$. 
The first one is local and measures the fraction of the norm of a polynomial when it is restricted to a measurable subset: 

\begin{defi}
    Let $\Omega\subset \partial \mathbb B$ be a measurable set and $Q\in \mathbb P_N^d$. We define the \emph{concentration} of $Q$ in $\Omega$ to be
\begin{align*}
    C_{N, \Omega}(Q):=\frac{ \binom{N+d}{N}\int_{\Omega} |Q(\zeta)|^2d\sigma(\zeta)}{\|Q\|_{\mathbb P^d_N}^2}.
\end{align*}
\end{defi}
Given a size $0< \omega < 1$ and a degree $N\in \mathbb N$, we are interested in the sets $\Omega\subset \partial \mathbb{B}$ with  $\sigma(\Omega) = \omega$ and the polynomials $Q\in \mathbb P^d_N$ that maximize the concentration $C_{N,\Omega}(Q)$.

\begin{rmk}\label{rmk:opt_superlevel}
Observe that if we are interested in maximizing the concentration of a polynomial $Q\in\mathbb P^d_N$ over all sets $\Omega$ with a prefixed size $\omega = \sigma(\Omega)$, it is always optimal to take $\Omega$ to be a superlevel set $\Omega_t$ of $|Q|^2$, i.e.
\begin{align}\label{eq:Omegat}
\Omega_t := \{\zeta\in \partial \mathbb B: |Q(\zeta)|^2>t\},
\end{align} 
with $t$ selected so that $\sigma(\Omega_t) = \omega$. With this choice $C_{N,\Omega_t} (Q) \ge C_{N,\Omega}(Q)$ for any other $\Omega$ of the same size $\omega$.
\end{rmk}

The other notion of the concentration of a polynomial that we will consider is global and provided by the generalized Wehrl entropy:
\begin{defi}\label{def:entropy}
Let $\Phi:[0,1]\to\mathbb R$ be a convex function. The \emph{generalized Wehrl entropy} for any $Q\in\mathbb P^d_N$ is defined as follows:
\begin{align*}
    S_{N,\Phi}(Q)=-\binom{N+d}{d}\mathlarger\int_{\partial\mathbb B} \Phi\left(\frac{|Q(\zeta)|^2}{\|Q\|_{\mathbb P^d_N}^2}\right)d\sigma(\zeta).
\end{align*}
\end{defi}
The generalized Wehrl entropy on homogeneous polynomials was initially studied by Lieb and Solovej in \cite{LiebSolovej16} (see also \cite{LiebSolovej14} for $\mathbb P^1_N$), where they proved that the reproducing kernels are minimizers for $S_{N,\Phi}$ for any convex function $\Phi$:
\begin{thm}[\cite{LiebSolovej16}] \label{thm:LS16}
Let $\Phi:[0,1]\to \mathbb R$ be a convex function. Then for any $Q\in\mathbb P^d_N$ with $\|Q\|_{\mathbb P^d_N}=1$ we have
\[
S_{N,\Phi}(Q) \ge S_{N,\Phi}(\zeta_1^N).\]
\end{thm} 

The estimate for the generalized Wehrl entropy also implies an estimate for  the (local) concentration $C_{N,\Omega}(Q)$:  Let us  consider  for any  $0<t_0<1$ the convex function
\begin{align*}
    \Phi_{t_0}(t) =\begin{cases}
        0 & \mbox{ if } 0\leq t < t_0,\\
        t-t_0 & \mbox{ if } t_0\leq t\leq 1.
    \end{cases}
\end{align*}
For any $Q\in\mathbb P^d_N$, let $\Omega_t$ denote its superlevel sets, as in \eqref{eq:Omegat}. Notice 
that 
$$S_{N,\Phi_{t_0}}(Q)=-C_{N, \Omega_{t_0}}(Q)+\binom{N+d}{d} t_0\sigma(\Omega_{t_0}).$$
Given $\omega\in(0,1)$, take $t_0$ such that $\omega=\min\{\sigma(\Omega_{t_0}), \sigma(\Omega_{t_0}^*)\}$, where  $\Omega^*_t=\{\zeta\in\partial\mathbb B: |\zeta_1|^{2N}>t\}$.
Then, combining Theorem~\ref{thm:LS16} for $\Phi_{t_0}$ with Remark~\ref{rmk:opt_superlevel}, one can infer:

\begin{cor}\label{cor:LS16}
Let $\omega \in(0, 1)$. For  any measurable  set $\Omega\subset \partial \mathbb B$ with $\sigma(\Omega) = \omega$ and any $Q\in \mathbb P^d_N$, it holds
\[
C_{N,\Omega}(Q) \le C_{N, \Omega^*} (\zeta_1^N),
\]
where $\Omega^* =\{\zeta\in\partial\mathbb B: |\zeta_1|^2>\big(1-\omega^{1/d}\big)\}$, i.e. $\sigma (\Omega^*) = \omega$.
\end{cor}

In dimension one, the uniqueness and stability of the optimizers for both inequalities are completely understood 
and can be found in \cites{KNOCT,Frank23} and \cite{GFOC25}, respectively.
In the case of higher dimensions, as it was first considered in \cite{LiebSolovej16}, Nicola, Riccardi and Till, in \cite{NRT25} (see also \cite{NRT25b} for an alternative proof), have proved  the uniqueness and stability of the optimizers for the generalized Wehrl entropy  $S_{N,\Phi}$ under the condition that $\Phi$ is non-linear in $(1-\epsilon, 1)$ for every $\epsilon$. This requirement precludes the application of their results in the case of entropies defined by convex functions such as $\Phi_{t_0}$. Therefore, it does not allow to infer the uniqueness or stability for the maximizers of the concentration operator.

Our first result addresses this shortcoming: In every dimension $d$, there is a threshold $0<\tilde{\omega}\le 1$ such that if $\omega < \tilde\omega$ the concentration operator for domains $\Omega$ such that $\sigma(\Omega) = \omega$ has only the multiples of reproducing kernels as maximizers, and they are stable, both in the function and in the domain. See Theorem~\ref{thm:concentration} for the precise statement. 

By rescaling this result as $N$ gets bigger, we obtain the stability of the concentration operators in the Bargmann-Fock space for domains of any size in the Euclidean space, see Section~\ref{sec:Fock}. This result was originally proved in \cite{GGRT}.

Our second main result concerns the stability of the minimizers of the generalized Wehrl entropy $S_{N,\Phi}$ for an arbitrary non-linear convex $\Phi$. We prove that given any such $\Phi$, there is a positive integer $N_\Phi$ (depending also on $d$) such that if $N\ge N_\Phi$, the unique minimizers of the entropy are the reproducing kernels  (see Theorem~\ref{thm:Wehrl}). Again, this result complements the results in \cites{NRT25, NRT25b} and passing to the limit as $N\to\infty$ we recover the stability results for the generalized Wehrl entropy in the Bargmann-Fock space proved originally in \cite{FNT25}. 

\begin{rmk} Observe that the previous articles, \cites{GFOC25, LiebSolovej16, NRT25, NRT25b},  present the results in terms of either  spaces of holomorphic polynomials of bounded degree or the symmetric irreducible  representations of the group $SU(d+1)$. The correspondence with the spaces of homogeneous polynomials where we are presenting our results  is introduced in Sections~\ref{subsec:irred} and~\ref{subsec:otherref}, respectively.
\end{rmk}

\subsection{Main results}
Before stating the previously announced results, we introduce some definitions on how to measure the distance of the polynomials and the domains to the optimal ones. 

For any $Q\in\mathbb P^d_N$ with $\|Q\|_{\mathbb P^d_N}=1$, we define its \emph{distance to the reproducing kernels} as:
\begin{equation}\label{eq:D_NQ}
    D_{N}(Q)=\min_{\eta \in\partial \mathbb B}\big\{\|Q-K_N(\cdot, \eta)\|_{\mathbb P^d_N}\big\}.
\end{equation}
In the case that we consider $Q\in\mathbb P^d_N$ with $\|Q\|_{\mathbb P^d_N}\neq 1$, we define $D_N(Q)$ as the distance to the reproducing kernels of the normalized polynomial $\frac{Q}{\|Q\|_{\mathbb P^d_N}}$.

The optimal domains for the concentration will be the superlevel sets of the reproducing kernels $K_N(\cdot,\eta)$, which are \emph{complex spherical caps} defined by:
\begin{align}\label{eq:Omega_zeta0}
    \Delta_t(\eta)=\big\{\zeta\in\partial\mathbb B: |\zeta\cdot\overline{\eta}|^N>t\big\},
\end{align} 
where $t\in(0,1)$ and $\eta\in\partial\mathbb B$. 
Notice that  $\sigma\big(\Delta_t(\eta)\big)=(1-t^{1/N})^d$.
We will measure the distance of any set $\Omega\subset\partial\mathbb B$ to the complex spherical caps of the same size in terms of the \emph{Fraenkel asymmetry} given by: 
\begin{equation*}
    \mathcal A_{\sigma}(\Omega):=\inf_{\eta\in\partial\mathbb B}
    \left\{\frac{\sigma\big(\Omega\backslash\Delta_t(\eta)\big)+
    \sigma\big(\Delta_t(\eta)\backslash\Omega\big)} {\sigma(\Omega)}\right\}, 
\end{equation*}
where $\sigma\big(\Delta_t(\eta)\big)=\sigma(\Omega)$,  i.e. $t=(1-\sigma(\Omega)^{1/d})^{N}$.

\begin{thm}\label{thm:concentration}
    For any $d\in\mathbb N$, there exist  constants $C>0$  and $\tilde \omega\in(0,1]$ (independent of $N$) such that for any measurable set $\Omega\subset\partial \mathbb B$ with $\sigma(\Omega)\in(0,\tilde\omega)$  and any $Q\in\mathbb{P}_N^d$ with $\|Q\|_{\mathbb{P}_N^d}=1$, there holds
    \begin{align}\label{eq:conc_polyn}
        D_{N}(Q)^2
        &\leq 
        \frac{C}{N^d\int_{\sigma(\Omega)}^{\tilde\omega}(1-s^{1/d})^Nds}
        \left(1-\frac{C_{N,\Omega}(Q)}{C_{N,\Omega^*}(\zeta_1^N)}\right),
        \\
        \label{eq:conc_set}
        \mathcal A_\sigma(\Omega)^2
        &\leq \frac{C}{\alpha\big(\sigma(\Omega)\big)}\left(1-\frac{C_{N,\Omega}(Q)}{C_{N,\Omega^*}(\zeta_1^N)}\right),
    \end{align}
    where
    \begin{align}\label{eq:Omega*}
        \Omega^*
        &=\big\{\zeta\in\partial\mathbb B: |\zeta_1|>\big(1-\sigma(\Omega)^{1/d}\big)^{N}\big\}.
        \\
        \label{eq:alpha_omega}
        \alpha(\omega)
        &={N^d\omega^2(1-\omega^{1/d})^{N-1}\int_{\omega}^{\tilde\omega}(1-s^{1/d})^N ds}.
\end{align}
\end{thm}

\begin{rmk}
    The threshold $\tilde\omega$ on the size of the sets in the statement of Theorem~\ref{thm:concentration} is due to the fact that, in the complex projective space,  balls are the isoperimetric domains only for small volumes (see the beginning of Section~\ref{subsec:monotonicity} and Proposition~\ref{prop:isop} for further details). Nevertheless, balls are always the (unique) isoperimetric domains in $\mathbb C\mathbb P^1$. Therefore, for $d=1$, we have $\tilde \omega=1$ and no restriction on the size of $\Omega$ appears, as  proved in \cite{GFOC25}. 
\end{rmk}

\begin{thm}\label{thm:Wehrl}
  Let $\Phi:[0,1]\to\mathbb R$ be a  convex, non-linear function and let $d\in\mathbb N$.  Then there exist  constants $C > 0$  and $N_\Phi\in\mathbb N$ (depending only on $\Phi$ and $d$) such that for any $N\geq N_\Phi$, the following holds: 
  Let $Q\in\mathbb P^d_N$ with $\|Q\|_{\mathbb P^d_N}=1$, then
  \begin{align*}
      D_N(Q)^2\leq C\left(S_{N,\Phi}(Q)-S_{N,\Phi}(\zeta_1^N)\right).
  \end{align*}
\end{thm}

\begin{rmk}
    The constant $N_\Phi$ is again related to the threshold for the balls as isoperimetric domains.  We will require that $\Phi$ is non-linear in $((1-\tilde\omega^{1/d})^N,1)$, which is satisfied for sufficiently large $N$ (once $d$ is fixed). It is worth to mention that there is a universal constant $T^*$, such that if $\Phi$ is non-linear in $(T^*,1)$ then $N_\Phi=1$ (see the proof of Theorem~\ref{thm:Wehrl}). Moreover, if $d=1$, we always have $\tilde \omega = 1$, see \cite{Oss}, and therefore $N_\Phi=1$, which agrees with our result in \cite{GFOC25}.
\end{rmk}

\begin{rmk}
    In Theorem~\ref{thm:concentration}, the ratio between the powers of $D_N(Q)$ or $\mathcal A_{\sigma}(\Omega)$ and of the deviation of the concentration  from the maximal value  is sharp.  
    The same happens in Theorem~\ref{thm:Wehrl} between the powers of $D_N(Q)$ and $S_{N,\Phi}(Q)-S_{N,\Phi}(\zeta_1^N)$.  Section~\ref{sec:optimality} deepens into these observations.
\end{rmk}

\subsection{Irreducible representations of \texorpdfstring{$SU(d+1)$}{SU(d+1)}}
\label{subsec:irred}
For the group $SU(d+1)$ with $d\ge 1$, the space $\mathbb P_N^d$ of homogeneous polynomials  of degree $N$  provides an explicit realization of a symmetric irreducible representation. Any element $\mathcal R\in SU(d+1)$ acts on a polynomial $P\in \mathbb P_N^d$ as $\pi(\mathcal R) P(\zeta) = P(\mathcal R^{-1}(\zeta))$. This finite-dimensional representation is irreducible and unitary for all $N\ge 0$. 

The previous representation is used in quantum mechanics, in whose language any positive-semidefinite  operator $\rho:\mathbb P_N^d\to\mathbb P_N^d$ with $\mathrm {Tr}(\rho) = 1$ defines a state.
If $\mathrm{rank}(\rho) = 1$ we have a pure state, otherwise we have a mixed or general state. In general we have that
there is an orthonormal basis $\{Q_j\}_j$ of $\mathbb P_N^d$ and constants $\lambda_j\ge 0$ such that $\sum_j \lambda_j = 1$ and
$$\rho(Q) = \sum_{j} \lambda_j \langle Q, Q_j\rangle_{\mathbb P^d_N} Q_j.$$ 
Given any $\eta\in \partial \mathbb B$, the pure state that corresponds to $\rho(Q)(\zeta) = \langle Q, K_N(\cdot,\eta)\rangle_{\mathbb P^d_N} K_N(\zeta,\eta)$ is called a \emph{coherent state}.

For a fixed  $\eta\in \partial \mathbb B$ and any  operator $\rho$  as above,  the associated \emph{Husimi function} $u_\rho:SU(d+1)\to \mathbb R$ is defined as 
 \begin{align}\label{eq:Husimi}
      u_\rho(\mathcal R) := \langle K_N(\cdot, \mathcal R \eta), \rho(K_N(\cdot,\mathcal R \eta))\rangle_{\mathbb P^d_N} = \sum_j \lambda_j |Q_j(\mathcal R\eta)|^2.
 \end{align}
Given a convex function $\Phi:[0,1]\to\mathbb{R}$ and a state $\rho$, with Husimi function $u_\rho$, its \emph{generalized Wehrl entropy} is defined as 
\begin{align}\label{eq:entropyop}
    \mathcal S_{N,\Phi}(\rho) = -\int_{SU(d+1)} \Phi(u_\rho(\mathcal R))\, d\mathcal R,
\end{align}
where $d\mathcal R$ is the Haar measure on $SU(d+1)$.
Notice that its value is independent of the choice of $\eta$ in the definition of the Husimi function. 
In the particular case when $\Phi(x) = x \log x$, this is the classical Wehrl entropy of the state $\rho$. 

In \cite{LiebSolovej16},  Lieb and Solovej showed that the entropy of a state is minimized at the coherent states. More precisely:
\begin{thm}[\cite{LiebSolovej16}]
\label{thm:LS16op}Let $\Phi:[0,1]\to \mathbb R$ be any convex function.
Then for any positive-semidefinite  operator $\rho:\mathbb P_N^d\to\mathbb P_N^d$ with $\mathrm {Tr}(\rho) = 1$   we have that
\[
\mathcal S_{N,\Phi}(\rho_0)\leq  \mathcal S_{N,\Phi}(\rho),
\]
where $\rho_0$ is  any coherent state. 
\end{thm}
Theorem~\ref{thm:Wehrl} provides a quantitative version of Lieb--Solovej's result in the case that $\rho$ is a pure state. Moreover, it implies that if $\Phi$ is nonlinear then the coherent states are the unique minimizers among the pure states, provided that $N\ge N_\Phi$. Later on we will see that Theorem~\ref{thm:Wehrl} can be extended to mixed states, see Theorem~\ref{thm:Wehrl_generalop}. It remains an open problem to prove an analogous result for non-symmetric irreducible representations of $SU(d+1)$ and more generally, other irreducible representations of other compact Lie groups.

\subsection{Other reformulations}
\label{subsec:otherref}

Our previous results can be formulated in the context of complex projective spaces and, using suitable coordinates, for holomorphic polynomials of bounded degree, generalizing the spaces considered in \cite{GFOC25} to higher dimensions. 
We start this section by introducing the setting, and we continue with the reformulation of the main results.

\subsubsection{Holomorphic sections on the complex projective space}
Let $M$ be a $d$-dimensional compact Kähler manifold. Consider a Hermitian holomorphic line bundle $(L, h) \to M$ whose curvature form, $\Theta$, is positive. This positivity means that the associated real form $\omega := i \Theta$ is a Kähler form on $M$, which defines a natural volume form $\frac{\omega^d}{d!}$.

For any power $N \ge 1$, the Hermitian metric $h$ on $L$ induces a metric $h^N$ on the line bundle $L^N$. We are interested in the vector space $H^0(M, L^N)$ of global holomorphic sections of $L^N$.
This space is endowed with the natural $L^2$ Hermitian inner product:
$$\langle s, t\rangle := \int_{M} h^N(s,t) \frac{\omega^d}{d!}, \quad \text{for } s, t \in H^0(M, L^N).$$

A fundamental example is when $M$ is the complex projective space $\mathbb{C}\mathbb{P}^d$ and $(L, h)$ is the hyperplane line bundle  equipped with the Fubini-Study metric.
In this specific case, the space of global holomorphic sections $H^0(\mathbb{C}\mathbb{P}^d, L^N)$ can be identified with the space of homogeneous polynomials of degree $N$ in $\mathbb{C}^{d+1}$ (see \cite{Shiffman15}) or, when restricted to affine coordinates, with polynomials in $d$ variables of degree $\le N$.

The $L^2$ inner product, when expressed using standard affine coordinates $z \in \mathbb{C}^d$, takes the explicit form:
\begin{align}\label{eq:innerprodsec}
    \langle s, t\rangle = C_{d,N} \int_{\mathbb C^d} \frac{p(z)\overline{q(z)}}{(1+|z|^2)^{N}} \frac{dz}{(1+|z|^2)^{d+1}},
\end{align}
where $p(z)$ and $q(z)$ are the coordinate representations of the sections $s$ and $t$ respectively, i.e., holomorphic polynomials in $d$ variables of degree smaller or equal than $N$,  $C_{d,N}$ is a normalization constant and $dz=\prod_{j=1}^d dz_j$, where $dz_j=dx_j dy_j$ and $z_j=x_j+\im y_j$.

\subsubsection{The space $\mathcal P^{d}_N$}
The previous observation makes it natural to consider the space $\mathcal P^{d}_N$ of polynomials of degree less than or equal to $N$ in $d$ complex variables endowed with a weighted  $L^2$ inner product as in \eqref{eq:innerprodsec}, which arises from 
the Fubini-Study metric. 
Any element of $\mathcal P_N^d$ admits the following expression:
\begin{align*}
    q(z)=\sum_{|\alpha|\leq N} a_\alpha z^\alpha, \quad \mbox{where } z\in\C^{d}, \ \alpha\in\mathbb N_0^{d} \mbox{ and }  \ a_\alpha\in\C.
\end{align*}

The Fubini-Study metric, given by
$$g_{jk}=\beta_d^{1/d}\frac{(1+|z|^2)\delta_{jk}-\bar z_jz_k}{(1+|z|^2)^2},$$
induces the following volume form:
\begin{align}\label{eq:defdm}
    dm(z)=\beta_d\frac{1}{(1+|z|^2)^{d+1}} dz.
\end{align}
We choose
\begin{align}\label{eq:gammad}
    \beta_d=\frac{2d}{|\partial \mathbb B^d|}=\frac{d!}{\pi^d}.
\end{align}
so $m(\mathbb C^d)=1$.
Finally, we require $\|1\|_{\mathcal P^d_N}=1$. 
Then, the inner product with which we endow $\mathcal P^{d}_N$ is:
\begin{align*}
    \langle p,q\rangle_{\mathcal P^d_N}=\binom{N+d}{d}\int_{\mathbb C^d} \frac{p(z)\overline{q(z)}}{(1+|z|^2)^N}dm(z).
\end{align*}

The space $\mathcal P^{d}_N$ with the previous inner product is also a reproducing kernel Hilbert space, with reproducing kernel
$$k_N(z,w)=(1+z\cdot\bar w)^N, \quad z, w\in\mathbb C^d.$$
Notice that reproducing kernels $k_N(\cdot,w)$ no longer have unit norm in $\mathcal P^d_N$, so we introduce their normalization:
\begin{align*}
    \kappa_N(z,w)=\frac{k_N(z,w)}{\|k_N(\cdot, w)\|_{\mathcal P^d_N}}=\frac{(1+z\cdot\bar w)^N}{(1+|w|^2)^N}.
\end{align*}

Finally, we notice that if $q(z)=Q(1,z)$ for any  $Q\in\mathbb P^d_N$, then $q\in\mathcal P^d_N$ and $\|q\|_{\mathcal P^d_N}=\|Q\|_{\mathbb P^d_N}$.
This correspondence allows us to rewrite Theorems~\ref{thm:concentration} and~\ref{thm:Wehrl} in these spaces.

\subsubsection{Main results}
In order to state the main theorems in $\mathcal P^d_N$ we must reintroduce the  notions of concentration and  the distance of the polynomials and the domains to the corresponding optimal ones in this setting. 

First of all, for any $\mathcal W\subset \mathbb C^d$ measurable  and any $q\in\mathcal P^d_N$, we define the concentration of $q$ in $\mathcal W$ as:
\begin{align*}
    C_{N,\mathcal W}(q)=\frac{\binom{N+d}{d}\int_{\mathcal W}\frac{|q(z)|^2}{(1+|z|^2)^N}dm(z)}{\|q\|_{\mathcal P^d_N}^2}.
\end{align*}
As in Definition~\ref{def:entropy}, for any $\Phi:[0,1]\to\mathbb R$ convex, the \emph{generalized Wehrl entropy} for any $q\in\mathcal P^d_N$ is:
\begin{align*}
    S_{N, \Phi}(q)=-\binom{N+d}{d}\mathlarger\int_{\mathbb C^d} \Phi\left(\frac{|q(z)|^2}{\|q\|_{\mathcal P^d_N}^2(1+|z|^2)^N}\right)dm(z).
\end{align*}

For any $q\in\mathcal P^d_N$ with $\|q\|_{\mathcal P^d_N}=1$, the \emph{distance to the (normalized) reproducing kernels} is defined by:
\begin{equation*}
    D_{N}(q)=\min\big\{\|q-e^{\im\theta}\kappa_N(\cdot, w)\|_{\mathcal P^d_N}:w \in\mathbb C^{d}, \theta\in[0,2\pi]\big\}.
\end{equation*}
Finally, the distance of any measurable set $\mathcal W\subset \mathbb C^d$  to the optimal domains, which in this case are  balls in $\mathbb C^d$, is given by the corresponding \emph{Fraenkel asymmetry}:
\begin{equation*}
    \mathcal A_{m}(\mathcal W)=\inf_{w\in\mathbb C^d}\left\{\frac{m\big(\mathcal W\backslash\mathbb B^d_r(w)\big)+m\big(\mathbb B^d_r(w)\backslash\mathcal W\big)}{m(\mathcal W)}\right\}, 
\end{equation*}
where  for any $w\in\mathbb C^d$, $r$ is chosen so $m\big(\mathbb B^d_r(w)\big)=m(\mathcal W)$.

\begin{thm}\label{thm:concentration_affine}
    For any $d\in\mathbb N$, there exist  constants $C>0$  and $\tilde \omega\in(0,1)$ (independent of $N$) such that for any measurable set $\mathcal W\subset \mathbb C^d$ with $m(\mathcal W)\in(0,\tilde\omega)$  and any $q\in\mathcal{P}_N^d$ with $\|q\|_{\mathcal{P}_N^d}=1$, there holds
    \begin{align}\label{eq:conc_polyn_affine}
        D_{N}(q)^2
        &\leq 
        \frac{C}{N^d\int_{m(\mathcal W)}^{\tilde\omega}(1-s^{1/d})^N ds}
        \left(1-\frac{C_{N,\mathcal W}(q)}{C_{N,\mathcal W^*}(1)}\right),\\
        \label{eq:conc_set_affine}
        \mathcal A_m(\mathcal W)^2&\leq \frac{C}{\alpha\big(m(\mathcal W)\big)}\left(1-\frac{C_{N,\mathcal W}(q)}{C_{N,\mathcal W^*}(1)}\right),
    \end{align}
    where
    $\mathcal W^*=\left\{z\in\mathbb C^d: |z|^2<\frac{m(\mathcal W)^{1/d}}{1-m(\mathcal W)^{1/d}}\right\}$
    and $\alpha(\omega)$ as in \eqref{eq:alpha_omega}.
\end{thm}

\begin{thm}\label{thm:Wehrl_affine}
  Let $\Phi:[0,1]\to\mathbb R$ be a  convex, non-linear function, and let $d\in\mathbb N$.  Then there exist  constants $C > 0$  and $N_\Phi\in\mathbb N$ (depending only on $\Phi$ and $d$) such that for any $N\geq N_\Phi$, the following holds: Let $q\in\mathcal P^d_N$ with $\|q\|_{\mathcal P^d_N}=1$, then
  \begin{align*}
      D_N(q)^2\leq C\big(S_{N,\Phi}(q)-S_{N,\Phi}(1)\big).
  \end{align*}
\end{thm}

This formulation of the main results allows us to recover the ones in the Bargmann-Fock space in \cite{FNT25} and \cite{GGRT}, respectively, by considering the limit $N\to\infty$ (see Section~\ref{sec:Fock}).

\subsection{Organization of  the article}
The rest of the article is organized as follows: In Section~\ref{sec:notation} we introduce  extra notation and preliminary results on the spaces.
Section~\ref{sec:superlevelsets} is devoted to  some
technical lemmas on the measure of the superlevel sets of functions related with the polynomials. Section~\ref{sec:proofs} contains the proofs of the main results. In Section~\ref{sec:further} we collect the following further results: the recovery of the results for the Fock space (Section~\ref{sec:Fock}), the sharpness of the main results (Section~\ref{sec:optimality}) and the generalization of the main results to general operators (Section~\ref{sec:generalop}). Since the proofs of all these results do not differ too much from those given in \cite{GFOC25}, most of the details are skipped, and only the main novelties are detailed.

\section{Notation and preliminaries}
\label{sec:notation}
Throughout the whole text,  $d\in\mathbb N$ will be fixed. Constants, usually denoted by $C$, may depend on $d$ but not on $N$. 
When $d=1$, all the results agree with those in \cite{GFOC25}.

Once $d$ is fixed,  $\mathbb B=\{\zeta\in\mathbb C^{d+1}: |\zeta|<1\}$ and  $\partial \mathbb B$ denotes its boundary, i.e., the  sphere of  (real) dimension  $2d+1$.
In the case of referring to the (complex) ball, or its boundary, in any other dimension, we will specify it. Namely, $\mathbb B^d=\{\zeta\in\mathbb C^{d}: |\zeta|<1\}$, and thus $\mathbb B=\mathbb B^{d+1}$.

Let $\zeta=(\zeta_1, \zeta_2,\dots, \zeta_{d+1})\in\partial\mathbb B$, with $\zeta_j\in\mathbb C$. We can express $\zeta$ in terms of the $2d+1$ real variables $\rho=(\rho_1,\dots,\rho_d)\in[0,1]^d$ and $\theta=(\theta_1,\dots,\theta_{d+1})\in[0,2\pi]^{d+1}$ as follows:
\begin{equation}\label{eq:zetapar}
\begin{aligned}
    \zeta_1&=e^{\im\theta_1}\sqrt{1-\rho_1},\\
    \zeta_2&=e^{\im\theta_2}\sqrt{\rho_1 (1-\rho_2)} ,\\
    &\vdots\\
    \zeta_{d}&=e^{\im\theta_{d}}\sqrt{\rho_1\rho_2\dots\rho_{d-1}(1-\rho_{d})},\\
    \zeta_{d+1}&=e^{\im\theta_{d+1}}\sqrt{\rho_1\rho_2\dots\rho_{d-1} \rho_{d}}.
\end{aligned}
\end{equation} 
With this parametrization, the measure on $\partial\mathbb B$ is given by: $$d\sigma(\zeta)=\frac{d!}{(2\pi)^{d+1}}\prod_{j=1}^{d}\rho_j^{d-j}d\rho_jd\theta, \quad d\theta=\prod_{k=1}^{d+1} d\theta_k.$$

For any $\rho\in[0,1]^d$ and $\theta\in [0,2\pi]^{d+1}$, we denote by $\zeta(\rho; \theta)$  the point in $\partial \mathbb B$ with the previous coordinates.
With some abuse of notation, given any function $U(\zeta)$, we may write $U(\rho; \theta)=U(\zeta(\rho; \theta))$.

\begin{lemma}\label{lem:onb}
    The set $\left\{\left(\frac{N!}{\alpha!}\right)^{1/2}\zeta^\alpha\right\}_{|\alpha|=N}$ is an orthonormal basis of $\mathbb P^d_N$. Furthermore, the norm \eqref{eq:norm} is equivalent to the Bombieri norm, defined for any $Q(\zeta)=\sum_{|\alpha|=N}a_\alpha\zeta^\alpha$ as:
    $$\|Q\|_{\mathbb P_N^d}^2=\sum_{|\alpha|=N }\frac{\alpha!}{N!}|a_\alpha|^2.$$
\end{lemma}

\begin{proof}
Using the parametrization in \eqref{eq:zetapar}, we firstly observe that  $|\zeta^\alpha(\rho;\theta)|^2$ does not depend on $\theta$. Therefore, provided $|\alpha|=N$,
\begin{align*}
 \int_{\partial\mathbb B} |\zeta|^{2\alpha} d\sigma(\zeta)
    =& d! \prod_{j=1}^{d} \left(\int_0^1 (1-\rho_j)^{\alpha_j} \rho_j^{N+d-j-(\alpha_1+\dots +\alpha_j)}d\rho_j\right)
    \\=& d! \prod_{j=1}^{d} \frac{\alpha!\big(N+d-j-(\alpha_1+\dots+\alpha_j)\big)!}{\big(N+d-(j-1)(\alpha_1+\dots+\alpha_{j-1})\big)!}
    =\frac{d! \prod_{j=1}^d \alpha_j!}{(N+d)!}=\frac{d!\ \alpha!}{(N+d)!}
\end{align*}
Consequently, $\|\zeta^\alpha\|_{\mathbb P_N^d}^2=\frac{\alpha!}{N!}$.

If $\alpha\neq \beta$, $\zeta^\alpha(\rho;\theta)\bar\zeta^\beta(\rho;\theta)$ depends on $e^{i\ell \theta_j}$ for some $j\in\{1,\dots,d+1\}$ and $\ell\in\mathbb Z\backslash\{0\}$, so $\int_0^{2\pi}e^{i\ell \theta_j} d\theta_j=0 $ and hence $\int_{\partial\mathbb B} \zeta^\alpha\bar\zeta^\beta d\sigma(\zeta)=0$.

Finally, the equivalence between norms follows from the orthogonality of the monomials and the values of their norm.
\end{proof}

\section{On the measure of the super-level sets}\label{sec:superlevelsets}
We start this section by  introducing some notation related with any $Q\in \mathbb{P}_N^d$:
\begin{align}
    \label{eq:Udef}
    U(\zeta)& = |Q(\zeta)|^2,
    \\
    \nonumber
    T&=\sup_{\zeta\in\mathbb \partial \mathbb B} U(\zeta),\\
    \nonumber
    \mu(t)& = \sigma\big(\{\zeta\in\partial\mathbb B: U(\zeta) > t\}\big)
\end{align}
If $Q(\zeta)=(\zeta\cdot\bar\eta)^N$ for some $\eta\in\partial\mathbb B$, then $T=1$ and
\begin{align}\label{eq:mu0}
    \mu(t)=(1-t^{1/N})^{d}=:\mu_0(t).
\end{align}

On the one hand, by \eqref{eq:supQ}, if $\|Q\|_{\mathbb P_N^d}= 1$, we can conclude that  $T \le 1$ and the equality is attained if and only if $Q(\zeta)=K_N(\zeta,\eta)=(\zeta\cdot\bar\eta)^N$ for some $\eta\in\partial \mathbb B$.
On the other hand, notice  that $\mu(0)=1$ and $\mu(t)=0$ for $t\geq T$.

Since $Q$ is holomorphic, we know that the level sets of $U(\zeta)$ cannot have positive measure in $\partial\mathbb B$ unless $Q(\zeta)$ is constant (see \cite{Rudin}*{Theorem 5.5.9}). Therefore, $\mu(t)$ is a strictly decreasing function in $[0,T]$ and therefore it  is invertible in $[0,1]$. We denote its inverse by $\mu^{-1}(s)$. In the case of $\mu_0$, we have:
\begin{align}\label{eq:mu0inv}
    \mu^{-1}_0(s)=(1-s^{1/d})^N.
\end{align}
Notice that if $\|Q\|_{\mathbb P_N^d}= 1$ then 
\begin{align}\label{eq:int_mu_muinv}
    \int_0^T\mu(t)dt=\int_{0}^1\mu^{-1}(s)ds=\frac{1}{\binom{N+d}{d}}.
\end{align}

In this section, we aim to study the function $\mu(t)$ for $Q\in\mathbb P_N^d$ with $\|Q\|_{\mathbb P_N^d}= 1$ and to bound the integrals of $\mu_0(t)-\mu(t)$ and  $\mu_0^{-1}(s)-\mu^{-1}(s)$ in suitable intervals.
In the comparison between $\mu$ and $\mu_0$, or their inverses, we will additionally use the following notation:
\begin{align}\label{eq:t*}
    t^*&=\sup\{t\in(0,T): \mu(t)\geq \mu_0(t)\},\\
    \label{eq:s*}
    s^*&=\inf\{s\in(0,1): \mu^{-1}(s)\geq \mu_0^{-1}(s)\}.
\end{align}
It is clear that $s^*=\mu^{-1}(t^*)=\mu^{-1}_0(t^*)$.

\subsection{Upper estimate for \texorpdfstring{$\mu(t)$}{µ(t)}.}
Our first result is analogous to  \cite{GFOC25}*{Lemma 2.1} and can be obtained similarly. We note that here we work with homogeneous polynomials in $\mathbb P^d_N$ and the measure $\sigma$ on $\partial \mathbb B$ instead of working with the  polynomials of bounded degree in $\mathcal P^d_N$ and the measure $m$ on $\mathbb C^d$.

\begin{lemma}\label{lem:super-level} 
For every $t_0\in (0,1)$, there exists a threshold $T_0\in [t_0,1)$ and a 
constant
$C_0=C_0(t_0\,)>0$ with the following property:
If $Q\in \mathbb{P}_N^d$ is such that $\|Q\|_{\mathbb P^d_N}=1$ and
$T\geq T_0$,
then
\begin{equation}
    \label{newestmu}
     \mu(t)\leq \big( 1+{C_0  
(1-T)}\big)\mu_0\left(\frac{t}T\right)\quad\forall t\in
(t_0,T),
\end{equation}
with $\mu_0$ as in \eqref{eq:mu0}.
\end{lemma}

\begin{proof}
The proof is split into five steps.

\noindent\textit{Step 1: Decomposition of $Q$.} 
We can assume without loss of generality  that $U(\zeta)$ attains its supremum in $\partial \mathbb B$ at $\zeta=(1, 0, \dots, 0)$, and that, in particular,  $Q(1,0,\dots, 0)=\sqrt{T}$. This, in addition, implies that $(\partial_{\zeta_k} Q)(1,0\dots, 0)=0$ for all $k\in\{2,\dots, d+1\}$.
We then write 
\begin{align*}
    Q(\zeta)=\sqrt{T}\zeta_1^N +\epsilon \tilde Q(\zeta), 
\end{align*}
where  $\epsilon^2=1-T$ and
\begin{align}\label{eq:defQ}
    \tilde Q(\zeta)=\sum_{\substack{|\alpha|=N\\\alpha_1\leq N-2}} b_\alpha \zeta^\alpha, \quad \|\tilde Q\|_{\mathbb P^d_N}=1.
\end{align}

With the previous decomposition, we obtain
\begin{equation}
  \label{eq:estimateP^2}
  |Q(\zeta)|^2\leq T|\zeta_1|^{2N}+ \epsilon^2 |\tilde Q(\zeta)|^2+ 2\sqrt{T}\epsilon \Re \bar\zeta_1^N \tilde Q(\zeta).
\end{equation}
Let $\zeta=(\zeta_1,\zeta')$ with  $\zeta'=(\zeta_2,\dots, \zeta_{d+1})$. 
Applying Lemma~\ref{lem:onb} and the Cauchy-Schwarz inequality, we estimate $\tilde Q(\zeta)$ in  \eqref{eq:defQ} as  follows:
\begin{equation}
\label{eq:estimateQ^2}
 |\tilde Q(\zeta)|^2
    \leq \left(\sum_{\substack{|\alpha|=N\\\alpha_1\leq N-2}}\frac{\alpha!}{N!}|b_\alpha|^2\right)\left(\sum_{\substack{|\alpha|=N\\\alpha_1\leq N-2}}\frac{N!}{\alpha!} |\zeta^\alpha|^{2}\right)
    \leq |\zeta|^{2N}-|\zeta_1|^{2N}-N|\zeta_1|^{2(N-1)}|\zeta'|^2.
\end{equation}
In particular, if $\zeta\in\partial\mathbb B$,
\begin{equation}
\label{eq:estimateQ^2_s}
|\tilde Q(\zeta)|^2\leq 1-|\zeta_1|^{2N}.
\end{equation}

Combining \eqref{eq:estimateP^2} and \eqref{eq:estimateQ^2_s}, we have  
\begin{align*}
    U(\zeta)\leq \epsilon^2+(T-\epsilon^2)|\zeta_1|^{2N}+2\sqrt{T}\epsilon \Re \bar\zeta_1^N \tilde Q(\zeta) \qquad \forall \zeta\in\partial\mathbb B.
\end{align*}
Then, provided that $T-\epsilon^2>0$, i.e. $T>\frac 12$,
\begin{align*}
    \mu(t)\leq \sigma\left(\left\{\zeta\in\partial\mathbb B: \frac{t-\epsilon^2}{T-\epsilon^2}|\zeta_1|^{-2N} -\frac{2\sqrt{T}\epsilon}{T-\epsilon^2} \Re \zeta_1^{-N}\tilde Q(\zeta)\leq 1 \right\}\right).
\end{align*}

Let $\tau\in[0,1]$ be an extra variable and define for any $\theta\in[0, 2\pi]^{d+1}$  and $\rho'\in[0,1]^{d-1}$ the function
\begin{align*}
    g_{\rho',\theta}(\rho_1,\tau)= \frac{t-\epsilon^2}{T-\epsilon^2}(1-\rho_1)^{-N} -\frac{2\sqrt{T}\epsilon}{T-\epsilon^2} \tau h\big((\rho_1,\rho');\theta\big).
\end{align*}
where $h(\zeta)=\Re \zeta_1^{-N} \tilde Q(\zeta)$.
We also introduce the sets
\begin{align*}
    E_\tau=\{\zeta(\rho;\theta)\in\partial\mathbb B: g_{\rho',\theta}(\rho_1,\tau)\leq 1\}.
\end{align*}
Note that $\mu(t)\leq \sigma(E_1)$.
In order to estimate $\sigma(E_1)$ 
we will see that the sets $E_\tau$ are star-shaped  domains with respect to the variable $\rho_1$ and  $E_\tau\subset E_{\tau'}$ for $\tau<\tau'$.

\noindent\textit{Step 2:  Estimates for $h(\rho;\theta)$.}
Throughout this step, we are going to use repeatedly the inequalities
\begin{align}\label{eq:ineqbinomial}
  \binom{N}j  \leq \frac{N^2}{j(j-1)}\binom{N}{j-2}, \quad \binom{N}j  \leq \frac{N^2}{j^2}\binom{N-2}{j-2}, \quad \mbox{ for } j\geq 2.  
\end{align}
Let $\tilde h(\zeta)=\zeta_{1}^{-N}\tilde Q(\zeta)$.
First of all, arguing as in \eqref{eq:estimateQ^2} and using \eqref{eq:ineqbinomial}, we obtain
\begin{align*}
    |\tilde h (\rho;\theta)|^2 
    &\leq \sum_{\alpha_1=0}^{ N-2}\binom{N}{\alpha_1} \left(\frac{\rho_1}{1-\rho_1}\right)^{N-\alpha_1}
    \leq \frac{N^2}{2}\sum_{j=2}^N\binom{N}{j-2} \left(\frac{\rho_1}{1-\rho_1}\right)^{j}\\
    &\leq  \frac{N^2}{2} \left(\frac{\rho_1}{1-\rho_1}\right)^2\left(1+\frac{\rho_1}{1-\rho_1}\right)^N=\frac{N^2}{2}\frac{\rho_1^2}{(1-\rho_1)^{N+2}}.
\end{align*}
Differentiating $\tilde h$ with respect to $\rho_1$ and using similar arguments, we infer
\begin{align*}
    |\partial_{\rho_1}\tilde h(\rho;\theta)|^2
    \leq& \sum_{\alpha_1=0}^{N-2}\left(\frac{N-\alpha_1}{2}\right)^2\binom{N}{\alpha_1} 
    \left(\frac{\rho_1}{1-\rho_1}\right)^{N-\alpha_1}\frac{1}{\rho_1^2(1-\rho_1)^2}
    \leq  \frac{N^2}{4}\frac{1}{(1-\rho_1)^{N+4}},
    \\
    |\partial_{\rho_1}^2\tilde h(\rho;\theta)|^2
    \leq& \sum_{\alpha_1=0}^{N-2}\left(\frac{N-\alpha_1}{4}\right)^2\binom{N}{\alpha_1} 
    \left(\frac{\rho_1}{1-\rho_1}\right)^{N-\alpha_1}\frac{(N-\alpha_1-2+4\rho_1)^2}{\rho_1^4(1-\rho_1)^4}
    \\
    \leq&  \frac{N^2}{8}\frac{1}{\rho_1^2(1-\rho_1)^6} \sum_{j=2}^N\left((j-2)^2+(4\rho_1)^2\right)\binom{N-2}{j-2}\left(\frac{\rho_1}{1-\rho_1}\right)^{j-2}\\ 
   \leq& \frac{C N^2}{\rho_1^2(1-\rho_1)^6} \Biggl(\frac{\rho_1^2}{(1-\rho_1)^{N-2}}+\frac{N\rho_1}{1-\rho_1}+N^2\sum_{j=4}^N\binom{N-2}{j-4}\left(\frac{\rho_1}{1-\rho_1}\right)^{j-2} \Biggr)
   \\ 
   \leq & \frac{C N^2}{\rho_1^2(1-\rho_1)^6} \left(\frac{\rho_1^2}{(1-\rho_1)^{N-2}}+\frac{1+N^2\rho_1^2}{1-\rho_1}+\frac{N^2\rho_1^2}{(1-\rho_1)^N} \right)\\
   \leq & CN^2\frac{1+N^2\rho_1^2}{\rho_1^2(1-\rho_1)^{N+6}}.
\end{align*}

The same inequalities apply to  $h=\Re \tilde h$, i.e.,
\begin{align}
    \label{eq:estimateh}
    |h(\rho;\theta)|&\leq \frac{N}{\sqrt 2}\frac{\rho_1}{(1-\rho_1)^{N/2+1}},\\
    \label{eq:estimatehr}
    |\partial_{\rho_1} h(\rho;\theta)|
    &\leq  \frac{N}{2}\frac{1}{(1-\rho_1)^{N/2+2}},\\
    \label{eq:estimatehrr}
    |\partial_{\rho_1}^2 h(\rho;\theta)|
    & \leq CN\frac{1+N\rho_1}{\rho_1(1-\rho_1)^{N/2+3}}.
\end{align}

\noindent\textit{Step 3: Star-shape of $E_\tau$.} 
The fact that $E_\tau$ is star-shaped with respect to the origin  in the $\rho_1$ variable follows from $g_{\rho',\theta}(0,\tau)<1$ and $\partial_{\rho_1} g_{\rho',\theta}(\rho_1,\tau)>0$ for any $\tau\in(0,1)$, $\rho'\in[0,1]^d$ and $\theta\in[0,2\pi]^{d+1}$. 
Indeed, $g_{\rho',\theta}(0,\tau)=\frac{t-\epsilon^2}{T-\epsilon^2}<1$.
Moreover, using \eqref{eq:estimatehr} and recalling that $T-\epsilon^2>0$ and $N\geq 2$, 
\begin{align}
    \label{eq:gr}
    \partial_{\rho_1} g_{\rho',\theta}(\rho_1,\tau)
    &=N\frac{t-\epsilon^2}{T-\epsilon^2}(1-\rho_1)^{-N-1}-\frac{2\sqrt{T}\epsilon}{T-\epsilon^2}\tau\partial_{\rho_1} h(\rho;\theta)\\\nonumber
    &\geq \frac{N}{T-\epsilon^2}(1-\rho_1)^{-N-1}\left({t-\epsilon^2}-{\sqrt{T}\epsilon}{}\tau(1-\rho_1)^{N/2-1}\right)\\\nonumber
    &\geq N(1-\rho_1)^{-N-1}\frac{t-\epsilon^2-\sqrt{T}\epsilon}{T-\epsilon^2}.
\end{align}
Given $t_0\in(0,1)$, let 
\begin{align}\label{eq:tildeT0}
    \tilde T_0=1-\frac{1+t_0-\sqrt{1+2t_0-t_0^2}}{4}.
\end{align}
Hence, if $T\geq \tilde T_0$, it holds
\begin{align}\label{eq:T0}
    \epsilon^2+\sqrt{T}\epsilon=1-T+\sqrt{T(1-T)}\leq\frac{t_0}{2},
\end{align}
and therefore, provided $t\in(t_0, T)$,
\begin{align}\label{eq:estimategr}
    \partial_{\rho_1} g_{\rho',\theta}(\rho_1,\tau)
    &\geq \frac{t_0}{T-\epsilon^2}\frac{N}{2}(1-\rho_1)^{-N-1}>0.
\end{align}

\noindent\textit{Step 4: Estimates for the radial distance of $E_\tau$.}
Let $\tau$, $\theta$ and $\rho'$ be as above. We define $r_\tau(\rho',\theta)\in(0,1)$ as the unique solution of $g_{\rho',\theta}(r,\tau)=1$.
It $\tau=0$, we know $r_0(\rho',\theta)=r_0$, where  $(1-r_0)^{-N}=\frac{T-\epsilon^2}{t-\epsilon^2}$. For any $\tau\in(0,1]$, by  \eqref{eq:estimategr} and the assumption $T-\epsilon^2<1$, we have
\begin{align*}
    g_{\rho',\theta}(\rho_1,\tau)
    &\geq g_{\rho',\theta}(0,\tau)+\frac{t_0 N}{2}\int_0^{\rho_1}  (1-r)^{-N-1} dr
    \geq \frac{t_0}{2}(1-\rho_1)^{-N}-1. 
\end{align*}
Hence, we infer that $r_\tau(\rho',\theta)\leq 1-\left(\frac{t_0}{4}\right)^{1/N}$.

During the rest of the proof,  we omit the  dependence on $\rho'$ and $\theta$ in $r_\tau$ and $g$.
By the implicit function theorem, it follows that 
\begin{align}
\label{eq:rs}
    \partial_\tau r_\tau=-\frac{\partial_s g(r_\tau,s)}{\partial_{\rho_1} g(r_\tau, s)}=\frac{2\sqrt{T}\epsilon}{T-\epsilon^2}\frac{h((r_\tau,\rho');\theta)}{\partial_{\rho_1} g (r_\tau, s)}.
\end{align}
Taking into account the estimates
 \eqref{eq:estimateh} and \eqref{eq:estimategr}, we conclude that
\begin{align}
\label{eq:estimaters}
    |\partial_\tau r_\tau|\leq \frac{\sqrt{2T}\epsilon}{t_0}r_\tau(1-r_\tau)^{N/2}.
\end{align}

We can differentiate  $r_\tau$ once more with respect to $\tau$ and since  $g$ is linear in $\tau$, we have
\begin{align*}
\begin{split}
    \partial_{\tau}^2 r_\tau
    &=-\frac{\partial_{\rho_1}\partial_{\tau} g(r_\tau,\tau)\partial_\tau r_\tau}{\partial_{\rho_1} g(r_\tau, \tau)}-\frac{\partial_{\tau} r_\tau}{\partial_{\rho_1} g(r_\tau, \tau)}\left(\partial_{\rho_1}^2g(r_\tau,\tau) \partial_\tau r_\tau+\partial_{\rho_1}\partial_{\tau}g(r_\tau,\tau)\right)
   \\& =-\frac{2\partial_{\rho_1}\partial_{\tau} g(r_\tau,\tau)\partial_\tau r_\tau+\partial_{\rho_1}^2 g(r_\tau,\tau)(\partial_\tau r_\tau)^2}{\partial_{\rho_1} g(r_\tau, \tau)}.
   \end{split}
\end{align*}
In order to estimate the remaining terms, we exploit  \eqref{eq:estimatehr} and \eqref{eq:estimatehrr} to infer
\begin{align*}
|\partial_{\rho_1}\partial_{\tau} g(\rho_1,\tau)|
    &=\frac{2\sqrt{T}\epsilon}{T-\epsilon^2}|\partial_{\rho_1} h(\rho;\theta)|
    \leq \frac{\sqrt{T}\epsilon}{T-\epsilon^2} N(1-\rho_1)^{-N/2-2},\\
    |\partial_{\rho_1}^2g(\rho_1,\tau)|
    &\leq 
    N(N+1)\frac{t-\epsilon^2}{T-\epsilon^2}(1-\rho_1)^{-N-2}+\frac{2\sqrt{T}\epsilon}{T-\epsilon^2}\tau\left|\partial_{\rho_1}^2 h(\rho;\theta)\right|\\
    &\leq 
    \frac{C}{T-\epsilon^2}N(1-\rho_1)^{-N-2}\frac{1+N\rho_1}{\rho_1}.
\end{align*}
These estimates together  with \eqref{eq:estimategr} and \eqref{eq:estimaters} finally give us  that
\begin{align*}
    |\partial_{\tau}^2 r_\tau|&\leq C \frac{T\epsilon^2}{t_0^2}r_\tau (1-r_\tau)^{N-1}\left(1+\frac{1}{t_0}\left(1+Nr_\tau\right)\right).
\end{align*}
Noticing that $Nr_\tau(1-r_\tau)^{N-1}$ is uniformly bounded from above, one obtains
\begin{align}\label{eq:estimaterss}
    |\partial_{\tau}^2 r_\tau|&\leq C \frac{T\epsilon^2}{t_0^3}r_\tau .
\end{align}

\noindent\textit{Step 5: Estimate for $m(E_1)$.}
The fact that $E_\tau$ is a star-shaped domain with respect to the variable $\rho_1$ allows us to express its $\sigma-$measure as: 
\begin{align*}
    M(\tau)=\sigma(E_\tau)=\frac{(d-1)!}{(2\pi)^{d+1}}\int_{[0,2\pi]^{d+1}}\int_{[0,1]^{d-1}} \big(r_\tau(\rho',\theta)\big)^d \prod_{j=2}^d \rho_j^{d-j}d\rho_j d\theta.
\end{align*}
Since $r_\tau(\rho', \theta)$ is uniformly bounded and so are  $|\partial_\tau r_\tau(\rho',\theta)|$ and $|\partial_{\tau}^2r_\tau(\rho',\theta)|$ according to \eqref{eq:estimaters} and \eqref{eq:estimaterss}, respectively, we can differentiate $M(\tau)$ under the integral so 
\begin{align}
\label{eq:M'}
    M'(\tau)
    &=\frac{d!}{(2\pi)^{d+1}}\int_{[0,2\pi]^{d+1}}\int_{[0,1]^{d-1}} \partial_\tau r_\tau r_\tau^{d-1} \prod_{j=2}^d \rho_j^{d-j}d\rho_j d\theta,\\
    \label{eq:M''}
    M''(\tau)
    &=\frac{d!}{(2\pi)^{d+1}}\int_{[0,2\pi]^{d+1}}\int_{[0,1]^{d-1}} \left(\partial_{\tau}^2 r_\tau r_\tau^{d-1}+(d-1)(\partial_\tau r_\tau)^2 r_\tau^{d-2}\right)\prod_{j=2}^d \rho_j^{d-j}d\rho_j d\theta.
\end{align}

We seek to estimate $M(1)$ using the following Taylor's formula: 
\begin{align}\label{eq:TaylorMs}
    M(1)=M(0)+M'(0)+\frac{M''(\tau)}{2},
\end{align}
for some  $\tau\in(0,1)$. 

We start estimating $M(0)$, which is given by  
\begin{align*}
    M(0)=r_0^d=\left(1-\left(\frac{t-\epsilon^2}{T-\epsilon^2}\right)^{1/N}\right)^{d}.
\end{align*}
Notice that \eqref{eq:T0} holds provided $T\geq \tilde T_0$, and therefore 
$\epsilon^2<t_0/2$. We also assume $T>t_0$ and consider $t\in(t_0, T)$.

Let $f(x)=1-\left(\frac{t-x}{T-x}\right)^{1/N}$, which is a non-decreasing function in $(0, t/2)$. 
Then, for any $x\in(0, t/2)$, we have
\begin{align*}
    f'(x)
    &=\frac{1}{N}\frac{T-t}{(t- x)(T- x)}\big(1-f(x)\big)
    \leq \frac{4}{NT}\left(\frac{T}{t}-1\right)\big(1-f(0)\big)\\
    & \leq  \frac{4}{t_0} \left(\left(\frac{T}{t}\right)^{1/N}-1\right)\left(\frac{t}{T}\right)^{1/N}
    =\frac{4}{t} f(0).
\end{align*}
Hence
\begin{align*}
    f(x)-f(0)=\int_0^x f'(\tilde x) d\tilde x\leq  \frac{4x}{t} f(0)
    \quad\mbox{ and }\quad
    f(x)\leq\left(1+\frac{4}{t_0}x\right)f(0).
\end{align*}
Noting that $M(0)=f(\epsilon^2)^d$ and  recalling \eqref{eq:mu0},  we finally conclude that
\begin{align*}
    M(0)
    &\leq \left(1+2\frac{\epsilon^2}{t_0}\right)^d\left(1-\left(\frac{t}{T}\right)^{1/N}\right)^d
    \leq \left(1+C\frac{\epsilon^2}{t_0}\right)\mu_0\left(\frac{t}{T}\right).
\end{align*}

Secondly, using   \eqref{eq:rs} and \eqref{eq:gr} in \eqref{eq:M'},
we obtain 
\begin{align*}
     M'(0)=
    &C_{t,T,N,d} \int_{[0,2\pi]^{d}}\int_{[0,1]^{d-1}}  {h\big((r_0,\rho');\theta\big)}\prod_{j=2}^d \rho_j^{d-j}d\rho_jd\theta.
\end{align*}
Since $h(\zeta)=\Re \bar\zeta_1^N\tilde Q(\zeta)$ with $\tilde Q$ given in \eqref{eq:defQ}, we observe that the previous integral is the real part of a sum of terms depending on $\theta_1$ as $e^{\im(\alpha_1-N)\theta_1}$ with $\alpha_1\in\{0,\dots, N-2\}$. Therefore, we conclude $M'(0)=0$.

Lastly,  applying  \eqref{eq:estimaters} and \eqref{eq:estimaterss} to \eqref{eq:M''}, we obtain that
\begin{align*}
    |M''(\tau)|
    &\leq C\int_{[0,2\pi]^{d}}\int_{[0,1]^{d-1}} \big(d|\partial_\tau r_\tau|^2 r_\tau^{d-2}+|\partial_{\tau}^2r_\tau|r_\tau^{d-1}\big) \prod_{j=2}^d \rho_j^{d-j}d\rho_jd\theta'\\
    & \leq C\frac{T\epsilon^2}{t_0^3} \int_{[0,2\pi]^{d}}\int_{[0,1]^{d-1}}  \big(r_\tau(\rho',\theta)\big)^d  \prod_{j=2}^d \rho_j^{d-j}d\rho_jd\theta\\
    &\leq  C\frac{\epsilon^2}{t_0^3} M(0),
\end{align*}
where in the last step we have used the monotonicity of $r_\tau$.

Combining the previous computations  with \eqref{eq:TaylorMs} leads to 
\begin{align*}
    M(1)&\leq \left(1+C\frac{\epsilon^2}{t_0^3}\right)\left(1+C\frac{\epsilon^2}{t_0}\right)\mu_0\left(\frac{t}{T}\right)\leq \left(1+C\frac{\epsilon^2}{t_0^3}\right)\mu_0\left(\frac{t}{T}\right).
\end{align*}
Therefore the result follows with $C_0=\frac{C}{t_0^3}$ and
\begin{align*}
    T_0=\max\left\{\frac12, t_0, \tilde T_0\right\}=\begin{cases}
        1-\frac{1+t_0-\sqrt{1+2t_0-t_0^2}}{4 } & \mbox{ if } \ t_0<\frac{2}{13}(4+\sqrt{3}),\\
        t_0 & \mbox{ if }\ t_0\geq \frac{2}{13}(4+\sqrt{3}).
    \end{cases}
\end{align*}
\end{proof}

\begin{rmk}
    A careful track of the $d$-dependence throughout the proof reveals that whereas $T_0$ only depends on $t_0$ and not on $d$, $C_0=\frac{cd^2 2^d}{t_0^3}$, where $c$ is a universal constant.  
\end{rmk}

\subsection{Monotonicity of \texorpdfstring{$\mu$}{μ} and isoperimetric inequalities}\label{subsec:monotonicity}

In order to continue with the study of $\mu$ and, in particular, by comparison with $\mu_0$, an isoperimetric inequality is exploited.

For convenience, and based on the identification of $\mathbb P^d_N$ and $\mathcal P^d_N$ (see Section~\ref{subsec:otherref}), in this section we  work at the level of $\mathcal P^d_N$.
In this case, 
\begin{align}\label{eq:mu_m}
    \mu(t)&
    = m\left(\left\{z\in\C^d: \frac{|Q(1,z)|^2}{(1+|z|^2)^N}>t \right\}\right),
\end{align}
with 
$dm(z)$ as  in \eqref{eq:defdm}.

One difficulty that arises immediately is the fact that  geodesic balls (or their complements)  in the complex projective space of complex dimension strictly bigger than one are no longer optimizers for the isoperimetric inequality for all volumes. This  forbids to control $\mu(t)$ in its whole domain, as noted in \cite{Frank23}*{Section 4.6}. In fact, as observed in \cite{Barbosa}*{Theorem 1.3}, there is a radius $r_0$ (depending on $d$) such that the geodesic balls of radius bigger than $r_0$ are no longer stable for the isoperimetric problem, meaning that there are small perturbations of them with the same volume and smaller perimeter. 

On the other hand, as noted by Nardulli in \cite{Nardulli09}, in a compact manifold  any solution $\Omega$  to the isoperimetric problem with small volumes $\omega$ is invariant under the action of the group of isometries that fixes the center of mass of $\Omega$.  The center of mass of a set $\Omega\subset \mathbb{CP}^d$ is defined as the point $p\in \mathbb{CP}^d$ that minimizes $\int_{\Omega} \mathrm d^2(x,p) dV(x)$, where the distance $\mathrm d$ and the volume $V$ are the ones induced by the Fubini-Study metric. If the diameter of the set $\Omega$ is small enough, the center of mass exists and it is unique.
We can assume that in affine coordinates, the center of mass is 0, and the set $\Omega$ must be invariant under rotations. Therefore,  for small volume, the solutions to the isoperimetric problem in the complex projective space are geodesic balls.

\begin{prop}\label{prop:isop}
    There exists $\tilde\omega\in (0,1]$ (depending on $d$) such  that for any $\mathcal W\subset \C^d$ with $m(\mathcal W)\leq \tilde\omega$, 
    the $(2d-1)$-dimensional Hausdorff measure (on $\mathbb C^d$ with the Fubini-Study metric) of its boundary, $|\partial\mathcal W|_{\mathcal H^{2d-1}_{FS}}$, satisfies
    \begin{align*}
       |\partial\mathcal W|_{\mathcal H^{2d-1}_{FS}}^2 \leq H(m(\mathcal W))
    \end{align*}
    where $H(x)=C_d x^{2-1/d}(1-x^{1/d})$ with $C_d=d^2|\mathbb B^{d}|^{1/d}=\frac{4\pi d^2}{(d!)^{1/d}}$
\end{prop}

The isoperimetric inequality for small volumes allows us  to still  argue as in \cite{GFOC25} for  sufficiently large values of $t$.
Before doing so, we introduce the following notation related with the threshold $\tilde \omega$:
\begin{align}\label{eq:tomega}
\begin{split}
    t^{\tilde\omega}&=\mu^{-1}(\tilde\omega),\\
    t_0^{\tilde\omega}&=\mu_0^{-1}(\tilde\omega)=(1-\tilde \omega^{1/d})^N.
\end{split}
\end{align}
Notice that $t_0^{\tilde\omega}$ decreases with $N$, so in particular
\begin{align}\label{eq:boundt0omega}
    t_0^{\tilde\omega}\leq 1-\tilde\omega^{1/d}.
\end{align}

\begin{lemma}\label{lem:on_mu}
    Let $Q\in\mathbb P_N^d$ with $\|Q\|_{\mathbb P_N^d}=1$ and $T\in(0,1)$.
    Then
    \begin{align}\label{eq:mu'}
       \mu'(t)\leq -\frac{d}{Nt}\mu(t)^{1-1/d}\big(1-\mu(t)^{1/d}\big), \quad t\in(t^{\tilde\omega},T).
    \end{align}
    Moreover 
    \begin{align*}
        &\frac{\mu(t)^{1/d}-\mu_0(t)^{1/d}}{t^{1/N}} \quad\mbox{ and }\quad\frac{\int_0^t \big(\mu(\tau)^{1/d}-\mu_0(\tau)^{1/d}\big)d\tau}{t^{1+1/N}}
    \end{align*}
    are non-increasing in $(t^{\tilde\omega},T)$.
\end{lemma}

\begin{proof}
    Let $v(z)=\frac{1}{2} \log \frac{|Q(1,z)|^2}{(1+|z|^2)^N}$ and $\nu(t)=m(\{ v(z)>t\})$.
    We seek to apply Theorem 1.1 in \cite{KNOCT}. 
    First of all,  we observe that the sets $\{z\in\mathbb C^d: v(z)\geq t\}$ are compact for $t>-\infty$ and that $m\big(\{z\in\mathbb C^d: \nabla_M v(z)=0\})=0$ (condition that replaces the assumption of $v$ to be Morse). Indeed, both facts can be verified at the level of $U(\zeta)=|Q(\zeta)|^2$ in $(\partial \mathbb B, d\sigma)$. The first one follows from the compactness of $\partial \mathbb B$ and the continuity of $U$, while the second one is based on \cite{Rudin}*{Theorem 5.5.9}.

    Secondly, we notice that 
    $$\Delta_M= 4\beta_d^{-1/d}(1+|z|^2) \sum_{j,k=1}^d (\delta_{jk}+\bar z_j z_k)\partial_{\bar z_j}\partial_{z_k},$$
    with $\beta_d$ as in \eqref{eq:gammad}.
    Then, 
    \begin{align}\label{eq:u_generalop}
        \Delta_M v(z)=-\frac{N}{2}4\beta_d^{-1/d}d.
    \end{align}
    
    Lastly, by Proposition~\ref{prop:isop}, $H(x)=C_d x^{2-1/d}(1-x^{1/d})$ for $x\leq \tilde\omega$, with   $C_d=\frac{4\pi d^2}{(d!)^{1/d}}$.

    Therefore, by \cite{KNOCT}*{Theorem 1.1}, 
    \begin{align}\label{eq:nu'}
        \nu'(t)\leq
        -\frac{2d}{N}
         \nu(t)^{1-1/d}\big(1-\nu(t)^{1/d}\big)
        , \quad t\in\left(\nu^{-1}(\tilde\omega), \frac{1}{2}\log T\right).
    \end{align}
    Since $\mu(t)=\nu(\frac{1}{2}\log{t})$,  estimate \eqref{eq:mu'} follows for $t\in(t^{\tilde\omega}, T)$. 

    The second part of the lemma can be inferred by computing derivatives and applying the previous inequality, which becomes an equality for $\mu_0$ in the whole range $(0,1)$.
\end{proof}

\begin{lemma}\label{lem:uniquet*}
    Let $Q\in\mathbb P_N^d$ with $\|Q\|_{\mathbb P_N^d}=1$ and $T\in(0,1)$.
    Let $t^*$ and $t^{\tilde\omega}_0$ be as in \eqref{eq:t*} and \eqref{eq:tomega}, respectively.
    If $t^*> t_0^{\tilde\omega}$,  then 
    $\mu(t)>\mu_0(t)$ for $t\in(t_0^{\tilde\omega}, t^*)$.
    In addition there exists a 
    constant $T^*\in(1/2,1)$ (independent of $N$) such that $t^*\leq T^*$.
\end{lemma}

\begin{proof}
    Let $t^*> t_0^{\tilde\omega}$.
    Arguing by contradiction, we infer that  $t^{\tilde\omega}> t_0^{\tilde\omega}$. Indeed, if $t^{\tilde\omega}< t_0^{\tilde\omega}$, we would have that $\mu(t_0^{\tilde\omega})<\mu_0(t_0^{\tilde\omega})$ while $\mu(t^*)=\mu_0(t^*)$, contradicting the second part of Lemma~\ref{lem:on_mu}. Hence, $\mu(t)>\mu_0(t)$ for $t\in (t_0^{\tilde\omega}, t^{\tilde\omega})$.
    
    Let us focus now on the interval $(t^{\tilde\omega}, t^*)$. By Lemma~\ref{lem:on_mu}, we conclude that  $\{t\in(t_0^{\tilde\omega},T): \mu(t)=\mu_0(t)\}$ is a connected interval, which is not empty since $t^*\in(t_0^{\tilde \omega}, T)$.
    We finally prove that its interior is empty, arguing  by contradiction. 
    
    Let us assume that $\mu(t)=\mu_0(t)$ for $t\in(t_1, t_2)\subset(t^{\tilde\omega}, T)$ and $T\in(1/2,1)$ and let $\nu(t)=\mu(e^{2t})$, which satisfies \eqref{eq:nu'}. 
    By hypothesis, $\nu(t)=(1-e^{2t/N})^d$ for $t\in I=\big(\frac{1}{2} \log(t_1), \frac{1}{2} \log(t_2)\big)$, so the equality is attained in  \eqref{eq:nu'} for $t\in I$. 
    Let $A_t=\{z\in\mathbb C^d: v(z)>t\}$. From the proof of \cite{KNOCT}*{Theorem~1.1}, we conclude that the equality is achieved provided $A_t$ is a disc, $|\nabla v|$ is constant on $\partial A_t$ and $\Delta_M v(z) = -2\pi Nd (d!)^{-1/d}$ for $z\in A_t$ and  $t\in I$.

    We firstly see that  $A_t$ are concentric balls for $t\in I$. Assume $A_t=\mathbb B^d_{r_t}(z_t)$.
    Since $v(z_t+\theta r_t)=t$ and $\nabla v(z_t+\theta r_t) = c_t \theta$ for all $\theta\in\partial\mathbb B^{d}$, we have
    \[1=\nabla v(z_t+\theta r_t) \cdot \partial_t(z_t+\theta  r_t)= c_t(\partial_t z_t \cdot \theta +\partial_t r_t).\]
    Since this must hold for any $\theta$,  we conclude that $\partial_t z_t=0$, i.e. $z_t$ does not depend on $t$. 
    We may assume without loss of generality that $z_t = 0$.

    Now, we note that   $\log\left|Q(1,z)\right|$ is harmonic and radial.  Since it is constant  on $\partial A_t$, it is constant in $A_t$. This means that $Q(1,z)$ is constant, or $|Q(\zeta)|=|\zeta_1|^N$, and hence $T=1$, against our hypothesis.

    In order to see the existence of a universal upper bound for $t^*$, let us apply Lemma~\ref{lem:super-level} with  $t_0=\frac{2}{13}(4+\sqrt{3})$, so $T_0= t_0$.  
    Therefore, in the case $t^*\geq t_0$ and $T\geq T_0=t_0$, we can apply  \eqref{newestmu} at $t=t^*$ with $C_0=C_0(t_0)$. Since $\mu(t^*)=\mu_0(t^*)$ we have
    \begin{align*}
        1-(t^*)^{1/N}\leq \big(1+C_0(1-T)\big)^{1/d}\left(1-\left(\frac{t^*}{T}\right)^{1/N}\right).
    \end{align*}
    This implies
    \begin{align*}
        t^*\leq \left(\frac{\big(1+C_0(1-T)\big)^{1/d}-1}{T^{-1/N}\big(1+C_0 (1-T)\big)^{1/d}-1}\right)^N.
    \end{align*}
    We can check that the right-hand side is an increasing function of $T$ in $(0,1)$, so
     \begin{align*}
        t^*\leq \left(\frac{CN}{1+C N}\right)^N.
    \end{align*}
    If $t^*<t_0=\frac{2}{13}(4+\sqrt{3})$ or $T<T_0$, then it also holds that 
    \begin{align*}
         t^*\leq T^*_{N}:=\max\left\{\frac{2}{13}(4+\sqrt{3}), \left(1-\frac{1}{1+CN}\right)^N \right\}<1.
    \end{align*}
    
    Finally, we notice that the term depending on $N$ at the right-hand side decreases with $N$, so
       \begin{align*}
        t^*\leq T^*:=\max\left\{\frac{2}{13}(4+\sqrt{3}), 1-\frac{1}{1+C}\right\}<1.
    \end{align*}

\end{proof}

\subsection{Concentration on superlevel sets}
We have already seen that  the Lieb-Solovej entropy inequality (Theorem~\ref{thm:LS16}) implies the  concentration inequality in Corollary~\ref{cor:LS16}. Rewriting it in terms of the measure of the superlevel sets of $U(\zeta)$ and $U_0(\zeta)=|\zeta_1|^{2N}$, one has:
\begin{prop}\label{prop:ineqLS}
    Let $Q\in\mathbb P_N^d$ with $\|Q\|_{\mathbb P_N^d}=1$.  Then for any $\hat s\in (0,1)$ and $\hat t\in (0, T)$,
    \begin{align}
    \label{eq:LS16conc} 
    \int_{0}^{\hat s}\big(\mu_0^{-1}(s)-\mu^{-1}(s)\big)ds&\geq 0,\\
    \nonumber
    \int_{0}^{\hat t}\big(\mu(t)-\mu_0(t)\big)dt&\geq 0.
    \end{align} 
\end{prop}
Notice that the second inequality can be obtained from \eqref{eq:LS16conc} after suitable changes of variables. Indeed, both inequalities are equivalent to the Lieb-Solovej entropy inequality  by the
Hardy-Littlewood-Polya dominating principle (see e.g. \cite{Hardy}).

In this section we seek to obtain a quantitative version of \eqref{eq:LS16conc}. More precisely, we aim to show an estimate of the form 
\begin{align}\label{eq:1-Tleq2}
    1-T\leq F_{N}(\hat s) \frac{\int_0^{\hat s}\big(\mu^{-1}_0(s)-\mu^{-1}(s)\big)ds}{\int_0^{\hat s}\mu^{-1}_0(s)ds},
\end{align}
where  $\hat s$ may be required to be small enough.
Notice that in terms  of $U$ and $U_0$, \eqref{eq:1-Tleq2} can be written as
\begin{align}\label{eq:1-Tleq}
    1-T\leq F_{N}\big(\mu(\hat t)\big) \left(1-\frac{\int_{\{U(\zeta)>\hat t\}}U(\zeta)d\sigma(\zeta)}{\int_{\{U_0(\zeta)>\hat t_0\}}U_0(\zeta)d\sigma(\zeta)}\right),
\end{align}
where  $\mu(\hat t)=\mu_0(\hat t_0)$.

An estimate like \eqref{eq:1-Tleq} implies the characterization of the maximizers of the concentration. Indeed \eqref{eq:supQ} implies that  $T=1$ if and only if $Q$ is a reproducing kernel (notice that the set of reproducing kernels $\{K(\cdot,\eta)\}_{\eta\in\partial\mathbb B}$ is invariant under the multiplication by unimodular constants).

Lemma~\ref{lem:super-level} implies the following  first estimate on the integral of $\mu_0(t)-\mu(t)$:

\begin{lemma}\label{lem:estimateintegral}
    For every $t_0\in(0,1)$, there is $T_0\in[t_0,1)$ and $C>0$ such that 
    if $Q\in\mathbb P^d_N$ satisfies $\|Q\|_{\mathbb P^d_N}=1$ with $T\geq T_0$, then for any $t\in(t_0, T)$ such that $t\geq t^*$ it holds
    \begin{align*}
        \int_{t}^T\big(\mu_0(\tau)-\mu(\tau)\big)d\tau\geq \frac{C}{N^{d}}(1-T)\left(1-\frac{t}{T}\right)^{d}.
    \end{align*}
\end{lemma}

\begin{proof}
    We apply Lemma~\ref{lem:super-level} with $t_0$. Then if $t\in(t_0, T)$ with $T\geq T_0$ we have
    \begin{align*}
      \mu_0(t)-\mu(t)\geq \mu_0(t)-\big(1+C_0(1-T)\big)\mu_0\left(\frac t T\right).
    \end{align*}
    Using the monotonicity of $\mu_0$, we start estimating  $\mu_0(t)-\mu_0\left(\frac{t}{T}\right)$ as follows:
    \begin{align*}
        \mu_0(t)-\mu_0\left(\frac{t}{T}\right)
        &=-\int_T^1\mu_0'\left(\frac{t}{\tau}\right)\frac{t}{\tau^2}d\tau
        =\frac{d }{N}t^{1/N}\int_T^1\left(\mu_0\left(\frac{t}{\tau}\right)\right)^{1-1/d}\tau^{-1-1/N}d\tau\\
        &\geq \frac{d }{N}t^{1/N}\left(\mu_0\left(\frac{t}{T}\right)\right)^{1-1/d}\int_T^1\tau^{-1-1/N}d\tau \\
         &=d (1-T^{1/N}) \left(\frac{t}{T}\right)^{1/N}\left(\mu_0\left(\frac{t}{T}\right)\right)^{1-1/d}.
    \end{align*}
    Taking into account that $1-T^{1/N}\geq\frac{1-T}{N}$, we have 
    \begin{align*}
        \mu_0(t)-\mu(t)
        &\geq (1-T)\left(\left(\frac{d}{N}+C_0\right)\left(\frac{t}{T}\right)^{1/N}-C_0\right)\left( \mu_0\left(\frac{t}{T}\right)\right)^{1-1/d},
     \end{align*}
     If $\frac t T\geq \left(\frac{C_0}{C_0+\frac d N}\right)^N$, then the right hand side is non-negative. In particular, this is satisfied for all values of $N$ if $\frac t T\geq \frac{C_0}{C_0+d}$. 
     Under this assumption and using that $x^{1/N}\geq 1+\frac{\log x}{N}$ for $x\in(0,1)$ we infer
     \begin{align*}
         \left(\frac{d}{N}+C_0\right)\left(\frac{t}{T}\right)^{1/N}-C_0
         \geq \frac{1}{N}\left(\frac{d}{C_0}-\log \left(1+\frac{d}{C_0}\right)\right)-\frac{d}{N^2}\log \left(1+\frac{d}{C_0}\right).
     \end{align*}
     Since $x>\log(1+x)$ for all $x>0$,  the first term in the right hand side is always positive
     and can absorb the second one for $N\geq N_0$ with $N_0$ sufficiently large. 
     Therefore, for suitable $C$ (depending on $d$) 
     \begin{align}\label{eq:diffpositive}
        \mu_0(t)-\mu(t)
        &\geq \frac{C}N(1-T)\left( \mu_0\left(\frac{t}{T}\right)\right)^{1-1/d}
     \end{align}
     provided $t\geq t_0$, $T\geq T_0$ and $\frac{t}{T}\geq \frac{C_0}{C_0+d}$.
     If $N<N_0$, considering $\frac{t}{T}$ big enough, the same result holds. We denote by $\tilde C_0<1$ the maximum of all the lower bounds for $\frac{t}{T}$.

     Let $t\in(t^*, T)$, so $\mu_0(\tau)-\mu(\tau)>0$ for $\tau\in(t, T)$. 
     Assume $t>t_0$  and $T>T_0$, with $T_0$ given by  Lemma~\ref{lem:super-level} for $t_0$.
     If $t\leq\tilde C_0 T$, using \eqref{eq:diffpositive}  for $\tau\geq \tilde C_0 T$ we obtain
     \begin{align*}
     I&=\int_t^T\big(\mu_0(t)-\mu(t)\big)d\tau\geq \int_{\tilde C_0 T}^T\big(\mu_0(\tau)-\mu(\tau)\big)d\tau
     \\&\geq \frac{C}{N} (1-T)T\int_{\tilde C_0}^1\big(\mu_0(\tau)\big)^{1-1/d}d\tau \geq \frac{C}{N^{d}}(1-T) .
     \end{align*}
     If $t>\tilde C_0 T$ then we can use  \eqref{eq:diffpositive} for $\tau\in(t, T)$ so
     \begin{align*}
     I 
     &\geq \frac{C}{N}(1-T)\int_{t}^T\left(\mu_0\left(\frac{\tau}{T}\right)\right)^{1-1/d}d\tau 
     \geq \frac{C}{N^{d}}(1-T)\int_t^T\left(1-\frac{\tau}{T}\right)^{d-1}d\tau
     \\&\geq\frac{C}{N^{d}}(1-T)\left(1-\frac{t}{T}\right)^{d-1}.
     \end{align*}
     Combining both cases, the result holds. 
     
\end{proof}

From now on, we focus on obtaining an estimate like \eqref{eq:1-Tleq2} with suitable $F_N$  for $\hat s<\tilde\omega$. We study separately, but with similar tools, two cases depending on the position of $s^*$ (in \eqref{eq:t*}) with respect to $\tilde\omega$.

\subsubsection{The case  $s^*<\tilde\omega$}

\begin{prop}\label{prop:case1}
     Let $Q\in\mathbb P^d_N$ with $\|Q\|_{\mathbb P^d_N}=1$. 
     Assume that $s^*\leq \tilde\omega$. Then for any {$\hat s\in(0,\tilde\omega)$} the estimate  in \eqref{eq:1-Tleq2} holds with
    \begin{align*}
      F_{N}(x)= \frac{C N^d}{\binom{N+d}{d}^2\int_x^{\tilde\omega}\mu_0^{-1}(s)ds}.
    \end{align*}
\end{prop}
\begin{proof}
    The result follows from the following Lemmas~\ref{lem:upper1} and~\ref{lem:lower1} together with the next observation: 
    Let $s^*$ and $t^*$ be as in \eqref{eq:t*}, then
    \begin{align*}
         \int_{t^*}^1\left(\mu_0(t)-\mu(t)\right)dt
        &=\int_{0}^{s^*}\left(\mu_0^{-1}(s)-\mu^{-1}(s)\right)ds,
    \end{align*}
    where for the last identity we have applied suitable changes of variables. 
\end{proof}

\begin{lemma}\label{lem:upper1}
    Under the assumptions of Proposition~\ref{prop:case1}, for any {$\hat s\in(0,\tilde\omega)$} it holds
    \begin{align}\label{eq:upper1}
        \int_{0}^{s^*}\left(\mu_0^{-1}(s)-\mu^{-1}(s)\right)ds \leq  \frac{1}{\binom{N+d}{d}^2}\frac{1}{\int_{\hat s}^{\tilde\omega}\mu_0^{-1}(s)ds}
        \left(1-\frac{\int_0^{\hat s}\mu^{-1}(s)ds}{\int_0^{\hat s}\mu^{-1}_0(s)}\right).
    \end{align}
\end{lemma}

\begin{proof}
    First of all, we observe that  the function 
    \begin{align}\label{eq:defrho}
        \rho(s)=\frac{\mu^{-1}(s)}{\mu_0^{-1}(s)},
    \end{align}
    is non-decreasing in $(0,\tilde\omega)$. 
    Indeed, the sign of $\rho'(s)$ is determined by the sign of 
    \begin{align*}
        (\mu^{-1})'(s)-\frac{(\mu_0^{-1})'(s)}{\mu_0^{-1}(s)}\mu^{-1}(s)
        =\frac{1}{\mu'\big(\mu^{-1}(s)\big)}+\frac{N}{d}\frac{s^{1/d-1}}{1-s^{1/d}}\mu^{-1}(s).
    \end{align*}
    Taking $t=\mu^{-1}(s)$, $\rho'(s)\geq0$  if and only if
    \begin{align*}
       \frac{1}{\mu'(t)}+\frac{Nt}{d}\frac{\mu(t)^{1/d-1}}{1-\mu(t)^{1/d}}\geq0,
    \end{align*}
    which holds by Lemma~\ref{lem:on_mu} for $t\in(\mu^{-1}(\tilde\omega),T)$.
    
    For any $0\leq s_1 <s_2\leq 1$, let $I(s_1, s_2)=\int_{s_1}^{s_2}\left(\mu_0^{-1}(s)-\mu^{-1}(s)\right)ds$.
    With this notation, \eqref{eq:upper1} can be written as
    \begin{align*}
        I(0, s^*)\leq \frac{1}{{\binom{N+d}{d}^2}}\frac{1}{\int_{\hat s}^{\tilde\omega}\mu_0^{-1}(s)ds}\frac{I(0, \hat s)}{\int_0^{\hat s}\mu_0^{-1}(s)ds}.
    \end{align*}

    \emph{Case 1: $\hat s<s^*<\tilde\omega$.}
    Taking into account the monotonicity of $\rho$ we obtain
    \begin{align*}
        I(0, \hat s)&\geq \big(1-\rho(\hat s)\big) \int_{0}^{\hat s} \mu_0^{-1}(s)ds,
        \\
        I(\hat s, s^*)&\leq \big(1-\rho(\hat s)\big)\int_{\hat s}^{s^*}\mu_0^{-1}(s)ds.
    \end{align*}
    Combining the previous inequalities, we infer
    \begin{align*}
        I(0, s^*)&=I(0, \hat s)+I(\hat s, s^*)
        \leq \left(1+\frac{\int_{\hat s}^{s^*}\mu_0^{-1}(s)ds}{\int_0^{\hat s}\mu_0^{-1}(s)ds}\right) I(0, \hat s)\\
        &\leq \frac{\int_{0}^{s^*}\mu_0^{-1}(s)ds}{\int_0^{\hat s}\mu_0^{-1}(s)ds} I(0, \hat s)
        \leq \frac{1}{\binom{N+d}{N}}\frac{I(0, \hat s)}{\int_0^{\hat s}\mu_0^{-1}(s)ds} .
    \end{align*}
    Since $\binom{N+d}{N}\int_{\hat s}^{\tilde\omega}\mu_0^{-1}(s)ds\leq 1$,  estimate \eqref{eq:int_mu_muinv} follows.

    \emph{Case 2:  $s^*<\hat s<\tilde\omega$.} 
    Now we have
    \begin{align*}
        I(0,s^*)=I(0,\hat s)-I(s^*, \hat s).
    \end{align*}
    By Proposition~\ref{prop:ineqLS}, 
    \begin{align*}
        I(0,\hat s)=I(0,\tilde\omega)-I(\hat s, \tilde\omega)\geq -I(\hat s, \tilde\omega).
    \end{align*}
    Using again the monotonicity of $\rho$, we have
    \begin{align*}
        -I(s^*,\hat s)&\leq \big(\rho(\hat s)-1\big)\int_{s^*}^{\hat s}\mu_0^{-1}(s)ds,\\
        -I(\hat s,\tilde\omega)&\geq \big(\rho(\hat s)-1\big)\int_{\hat s}^{\tilde\omega}\mu_0^{-1}(s)ds.
    \end{align*}
    Combining the previous inequalities, we conclude
    \begin{align*}
        I(0,s^*)&\le\left(1+\frac{\int_{s^*}^{\hat s}\mu_0^{-1}(s)ds}{\int_{\hat s}^{\tilde\omega}\mu_0^{-1}(s)ds}\right)I(0,\hat s)\leq \frac{\frac{1}{\binom{N+d}{d}}}{\int_{\hat s}^{\tilde\omega}\mu_0^{-1}(s)ds}I(0,\hat s)\\
        &\leq\frac{1}{\binom{N+d}{d}^2}\frac{1}{\int_{\hat s}^{\tilde\omega}\mu_0^{-1}(s)ds}\frac{I(0, \hat s)}{\int_0^{\hat s}\mu_0^{-1}(s)ds}.
    \end{align*}
\end{proof}

The following lemma, which  is a consequence of Lemma~\ref{lem:estimateintegral} and does not make use of the isoperimetric inequality, holds independently of  the relative position of $t^*$ with respect to $t^{\tilde\omega}_0$.

\begin{lemma}\label{lem:lower1}
    There exists a constant $C>0$ such that for any  $Q\in\mathbb P_N^d$ with $\|Q\|_{\mathbb P_N^d}=1$ and $T<1$ it holds
    \begin{align}
        \int_{t^*}^1\big(\mu_0(t)-\mu(t)\big)dt
        \geq \frac{C}{N^d}(1-T).
    \end{align}
\end{lemma}

\begin{proof}
Let $T^*$  be the upper bound of $t^*$  as in Lemma~\ref{lem:uniquet*}. If  $T>T^*$, then
    \begin{align*}
        I=\int_{t^*}^1\big(\mu_0(t)-\mu(t)\big)dt\geq \int_{T^*}^T\big(\mu_0(t)-\mu(t)\big)dt.
    \end{align*}
    Now we apply Lemma~\ref{lem:estimateintegral} with $ t_0=T^*$. 
    Then there exists $T_0\geq T^*$   and $C$  such that if $T\geq T_0$, we have 
   \begin{align*}
        I \geq \frac{C}{N^{d}}(1-T)\left(1-\frac{T^*}{T}\right)^{d}.
    \end{align*}
    Assuming $T\geq T_1=\frac{1+T_0}{2}$ we conclude
    \begin{align*}
        I \geq \frac{C}{N^{d}}(1-T)\left(1-\frac{T^*}{T_1}\right)^{d}\geq \frac{C}{N^d}(1-T).
    \end{align*}
    
    So far we have assumed that $T>T_1$. If this is not the case, 
     we have 
    \begin{align}\label{eq:farfrom1}
    \begin{split}
    I&\geq \int_{T}^1 \big(\mu_0(t)-\mu(t)\big)dt
        =\int_{T}^1\left(1-t^{1/N}\right)^ddt
        \\&\geq (1-T^{1/N})^d(1-T)\geq (1- T_1^{1/N})^d(1-T)\geq\frac{(1-T_1)^d}{N^d}(1-T).
        \end{split}
    \end{align}
\end{proof}

\subsection{The case \texorpdfstring{$s^*>{\tilde\omega}$}{s*>w}}

\begin{prop}\label{prop:case2}
     Let $Q\in\mathbb P^d_N$ with $\|Q\|_{\mathbb P^d_N}=1$. 
     Assume that $s^*> \tilde\omega$. Then for any {$\hat s\in(0,\tilde\omega)$} the estimate in \eqref{eq:1-Tleq2} holds with
    \begin{align*}
      F_{N}(x)= \frac{C N^d}{\binom{N+d}{d}}.
    \end{align*}
\end{prop}
\begin{proof}
    The desired estimate is a consequence of Lemmas~\ref{lem:upper2} and~\ref{lem:lower2} and the following estimate
    \begin{align}\label{eq:est_int_omega}
        \int_0^{\tilde\omega} \big(\mu_0^{-1}(s)-\mu^{-1}(s)\big)ds
        \geq \int_{t_0^{\tilde\omega}}^1\big(\mu_0(t)-\mu(t)\big)dt.
    \end{align}
    This follows from  
    \begin{align*}
        \int_0^{\tilde\omega} \big(\mu_0^{-1}(s)-\mu^{-1}(s)\big)ds
        = \int_{t_0^{\tilde\omega}}^1\mu_0(t)dt-\int_{t^{\tilde\omega}}^1\mu(t)dt+ (t_0^{\tilde\omega}-t^{\tilde\omega})\omega,
    \end{align*}
    where $t^{\tilde\omega}$ is given in \eqref{eq:tomega}.
    By Lemma~\ref{lem:uniquet*} we know $t^{\tilde\omega}_0> t^{\tilde\omega}$ and by the monotonicity of $\mu$ we have
    \begin{align*}
        \int_{t^{\tilde\omega}}^{t^{\tilde\omega}_0}\mu(t)dt\leq \mu(t^{\tilde\omega})(t^{\tilde\omega}_0-t^{\tilde\omega})=\tilde\omega (t^{\tilde\omega}_0-t^{\tilde\omega}).
    \end{align*}
    Therefore,  \eqref{eq:est_int_omega} is proved.
\end{proof}

\begin{lemma}\label{lem:upper2}
    Under the assumptions of Proposition~\ref{prop:case2}, for any {$\hat s\in(0,\tilde\omega)$} it holds
    \begin{align*}
        \int_{0}^{\tilde\omega}\big(\mu_0^{-1}(s)-\mu^{-1}(s)\big)ds \leq  \frac{1}{\binom{N+d}{d}}
        \left(1-\frac{\int_0^{\hat s}\mu^{-1}(s)ds}{\int_0^{\hat s}\mu^{-1}_0(s)ds}\right).
    \end{align*}
\end{lemma}

\begin{proof}
    We argue as in the first case in the proof of Lemma \eqref{lem:upper1}, with $\tilde \omega$ playing the role of $s^*$. Actually, using the notation introduced in that proof and exploiting the fact that the function
     $\rho$ given by \eqref{eq:defrho} is non-decreasing in $(0, \omega)$, we obtain
    \begin{align*}
        I(0, \hat s)&\geq \big(1-\rho(\hat s)\big) \int_{0}^{\hat s} \mu_0^{-1}(s)ds,
        \\
        I(\hat s, \tilde\omega)&\leq \big(1-\rho(\hat s)\big)\int_{\hat s}^{\tilde\omega}\mu_0^{-1}(s)dt.
    \end{align*}
    Combining the previous inequalities, we infer
    \begin{align*}
        I(0, \tilde\omega)&=I(0, \hat s)+I(\hat s, \tilde\omega)
        \leq \left(1+\frac{\int_{\hat s}^{\tilde\omega}\mu_0^{-1}(s)ds}{\int_0^{\hat s}\mu_0^{-1}(s)ds}\right) I(0, \hat s)\\
        &\leq \frac{\int_{0}^{\tilde\omega}\mu_0^{-1}(s)ds}{\int_0^{\hat s}\mu_0^{-1}(s)ds} I(0, \hat s)
        \leq \frac{1}{\binom{N+d}{N}}\frac{I(0, \hat s)}{\int_0^{\hat s}\mu_0^{-1}(s)ds} .
    \end{align*}
\end{proof}

\begin{lemma}\label{lem:lower2}
    There exists a constant $C>0$ such that for any  $Q\in\mathbb P_N^d$ with $\|Q\|_{\mathbb P_N^d}=1$, $T<1$ it holds
    \begin{align}
        \int_{t_0^{\tilde\omega}}^1\big(\mu_0(t)-\mu(t)\big)dt
        \geq \frac{C}{N^d}(1-T).
    \end{align}
\end{lemma}

\begin{proof}
    We start recalling \eqref{eq:boundt0omega} and applying Lemma~\ref{lem:estimateintegral} with $t_0=1-\tilde\omega^{1/d}$.
    If $T\geq T_1=\frac{1+T_0}{2}$, with $T_0$ given by Lemma~\ref{lem:estimateintegral}, then
    \begin{align*}
        \int_{t_0^{\tilde\omega}}^1\big(\mu_0(t)-\mu(t)\big)dt
        &\geq \int_{t_0}^T\big(\mu_0(t)-\mu(t)\big)dt\geq \frac{C}{N^d} (1-T)\left(1-\frac{t_0}{T_1}\right)^{d}
        \geq\frac{C}{N^d}(1-T).
    \end{align*}
    If $T\leq T_1$ we simply argue as in \eqref{eq:farfrom1}.
\end{proof}

\section{Proof of the main results}\label{sec:proofs}

The main results are shown primarily from the propositions of the previous section together with the following lemma, which characterizes $D_N(Q)$, defined in \eqref{eq:D_NQ},  in terms of $T=\sup_{\zeta\in\partial\mathbb B}|Q(\zeta)|^2$.

\begin{lemma}\label{lem:valuemin}
    Let $Q\in\mathbb{P}_N^d$ with $\|Q\|_{\mathbb P^d_N}=1$.
    Then 
    \begin{align*}
        D_{N}(Q)^2=2(1-\sqrt{T}).
    \end{align*}
\end{lemma}

\begin{proof}
For any $\eta\in\partial \mathbb B$ we have
    $$\|Q-K_N(\cdot, \eta)\|_{\mathbb P^d_N}^2 \leq 2 \big(1-\Re Q(\eta)\big).$$
By rotating $\eta$ suitably, $Q(\eta)$ can be taken real. 
The minimum is then attained if $|Q|^2(\eta)=T$.
\end{proof}

\begin{proof}[Proof of Theorem~\ref{thm:concentration}, estimate \eqref{eq:conc_polyn}]
As observed in Remark~\ref{rmk:opt_superlevel}, 
\begin{align*}
    C_\Omega(Q)\leq \int_\Omega |Q(\zeta)|^2d\sigma(\zeta)=\int_\Omega U(\zeta)d\sigma(\zeta)\leq \int_{\{U(\zeta)>\hat t\}} U(\zeta)d\sigma(\zeta)
\end{align*}
with $\hat t=\mu^{-1}\big(\sigma(\Omega)\big)$. 
By Propositions~\ref{prop:case1} and~\ref{prop:case2}, estimate \eqref{eq:1-Tleq} holds with 
\begin{align*}
      F_{N}(x)= \frac{C N^d}{\binom{N+d}{d}^2\int_x^{\tilde\omega}\mu_0^{-1}(s)ds}.
\end{align*}
Indeed, notice that this function agrees with the one in Proposition~\ref{prop:case1} and, by \eqref{eq:int_mu_muinv},  bounds from above the one in Proposition~\ref{prop:case2}.
The previous observations imply
\begin{align}\label{eq:ep_est}
        1-T\leq \frac{CN^d}{\binom{N+d}{d}^2\int_x^{\tilde\omega}\mu_0^{-1}(s)ds}\left(1-\frac{C_{N,\Omega}(Q)}{C_{N,\Omega^*}(\zeta_1^N)}\right).
\end{align}
Combining this result with Lemma~\ref{lem:valuemin} and   
\begin{align}\label{eq:binompot}
    C^{-1} N^d\leq \binom{N+d}{d}\leq CN^d, \quad C=C(d)>1,
\end{align}
 the estimate in \eqref{eq:conc_polyn} follows.
\end{proof}

\begin{proof}[Proof of Theorem~\ref{thm:concentration}, estimate \eqref{eq:conc_set}]    
    Let us assume without loss of generalization, as in the proof of Lemma~\ref{lem:super-level}, that 
    \begin{align*}
        Q(\zeta)=\sqrt{T}\zeta_1^N+\epsilon \tilde Q(\zeta),
    \end{align*}
    with $\|\tilde Q\|_{\mathbb P_N^d}=1$ and $\epsilon^2=1-T$, bounded by \eqref{eq:ep_est}.
    As previously, let $U_0(\zeta)=|\zeta_1|^{2N}$.

    \noindent\emph{Step 1: General considerations.}
    Let $t\in(0,T)$ be such that $\mu(t)=\sigma(\Omega)=\omega$ and let $\Omega_t=\{\zeta\in\partial \mathbb B: U(\zeta)\geq t\}$. 
    Let
    \begin{align*}
        \mathcal D(\Omega)=\int_{\Omega_{t}}U(\zeta)d\sigma(\zeta)-\int_\Omega U(\zeta)d\sigma(\zeta).
    \end{align*}
    By the qualitative version of \eqref{eq:conc_polyn} 
    \begin{align}\label{eq:dOmega_est}
    \begin{split}
        \mathcal D(\Omega)
        &\leq \int_{\Omega^*} U_0(\zeta)d\sigma(\zeta)-\int_\Omega U(\zeta)d\sigma(\zeta)
        \\&={\binom{N+d}{d}\int_0^{\omega}(1-s^{1/d})^Nds}\left(1-\frac{C_{N, \Omega}(Q)}{C_{N,\Omega^*}(\zeta_1^N)}\right).
    \end{split}
    \end{align}

    On the other hand, for any $\zeta\in\partial\mathbb B$ we have
    \begin{align}\label{eq:u-Tu0}
        TU_0(\zeta)-U(\zeta)\leq 2\epsilon\sqrt{T}|\Re \bar\zeta_1\tilde Q(\zeta)|\leq 2 \epsilon
    \end{align}
    and then
    \begin{align*}
        \left\{\zeta\in\partial\mathbb B: U_0(\zeta)\geq \frac{t+2\epsilon}{T}\right\}\subset \Omega_t \subset \left\{\zeta\in\partial\mathbb B: U_0(\zeta)\geq \frac{t-2\epsilon}{T}\right\}.
    \end{align*}
    Hence,
    \begin{align*}
    \mu_0\left(\frac{t+2\epsilon}{T}\right)\leq \omega\leq \mu_0\left(\frac{t-2\epsilon}{T}\right),
    \end{align*}
    which implies
    \begin{align*}
        t_0T-2\epsilon\leq t\leq t_0T+2\epsilon.
    \end{align*}
    where
    $t_0=\mu_0^{-1}(\omega)=\left(1-\omega^{1/d}\right)^N$.
    This  implies
    \begin{align}\label{eq:inclusions}
        \left\{\zeta\in\partial\mathbb B: U_0(\zeta)\geq t_0+\frac{4\epsilon}{T}\right\}\subset \Omega_t \subset \left\{\zeta\in\partial\mathbb B: U_0(\zeta)\geq t_0-\frac{4\epsilon}{T}\right\}.
    \end{align}

    \noindent\emph{Step 2: The set $\Lambda$.}
    By Brenier Theorem, it is known that there exists a transport map $\mathcal T: \Omega_t\backslash\Omega\to\Omega\backslash \Omega_t,\ \zeta\to\mathcal T\zeta$ sending $\sigma|_{\Omega_t\backslash\Omega}$ to 
    $\sigma|_{\Omega\backslash\Omega_t}$.
    Let
    \begin{align}\label{eq:gamma}
         \gamma^2=\frac{C}{\binom{N+d}{d}\int_\omega^{\tilde\omega}\mu_0^{-1}(s)ds}\left(1-\frac{C_{N,\Omega}(Q)}{C_{N,\Omega^*}(\zeta_1^N)}\right),
    \end{align}
    so $\epsilon\leq\gamma$ by \eqref{eq:ep_est}.
    We define the subset 
    \begin{align*}
        \Lambda=\big\{\zeta\in\Omega_t\backslash\Omega: 
        |(\mathcal T \zeta)_1|^{2N}\geq |\zeta_1|^{2N}+10\gamma\big\}.
    \end{align*}
    Notice that in this subset, by \eqref{eq:u-Tu0}, we have
    \begin{align*}
        U(\zeta)-U(\mathcal T\zeta)
        &\geq T|U_0(\zeta)-U_0(\mathcal T\zeta)|-4\epsilon\geq 10T \gamma-4\epsilon.
    \end{align*}
    Under the assumption $T\geq \frac 12$, one obtains
    \begin{align}\label{eq:UinL}
        U(\zeta)-U(\mathcal T\zeta)\geq \gamma \quad \forall \zeta\in\Lambda.
    \end{align}
    Therefore
    \begin{align*}
        \gamma\sigma(\Lambda)\leq \int_{\Lambda}\big(U(\zeta)-U(\mathcal T\zeta)\big)d\sigma(\zeta)
        \leq \int_{\Omega_t\backslash\Omega}\big(U(\zeta)-U(\mathcal T\zeta)\big)d\sigma(\zeta)=\mathcal D(\Omega).
    \end{align*}
    Combining this with the estimate \eqref{eq:dOmega_est} and \eqref{eq:binompot} we obtain
    \begin{align}\label{eq:s(L)}
        \sigma(\Lambda)\leq \frac{1}{\gamma}{ \binom{N+d}{d}\int_0^{\omega}(1-s^{1/d})^Nds}\left(1-\frac{C_{N, \Omega}(Q)}{C_{N,\Omega^*}(\zeta_1^N)}\right).
    \end{align}

   \noindent \emph{Step 3: Estimate for $\sigma(\Omega_t\backslash \Omega)$.}
   Notice that 
    \begin{align}\label{eq:subsetsplit}
        \sigma(\Omega_t\backslash \Omega)=\sigma(\Lambda)+\sigma\big((\Omega\backslash\Omega_t)\backslash\mathcal T(\Lambda)\big)
    \end{align}
    and that the first term in the right hand side has been already estimated in \eqref{eq:s(L)}.
    Therefore, we focus on the set $(\Omega\backslash\Omega_t)\backslash\mathcal T(\Lambda)$.
    
    By the definition of the set $\Lambda$, we have
    \begin{align*}
        (\Omega\backslash\Omega_t)\backslash\mathcal T(\Lambda)
        &\subset\left\{ \zeta\in\Omega\backslash \Omega_t: |\zeta_1|^{2N}\leq 10 \gamma +|\zeta_1'|^{2N} \mbox { for some } \zeta_1'\in \Omega_t\right\}
        \\&\subset \left\{\zeta\in\partial\mathbb B:  U_0(\zeta)\leq 10\gamma+t_0-\frac{4\epsilon}{T}\right\}, 
    \end{align*}
    where the last inclusion holds by \eqref{eq:inclusions}. This also implies
    that
    \begin{align*}
     (\Omega\backslash\Omega_t)\backslash\mathcal T(\Lambda) \subset \left\{\zeta\in\partial\mathbb B: t_0+\frac{4\epsilon}{T}\leq U_0(\zeta)\leq 10\gamma+t_0-\frac{4\epsilon}{T}\right\}
    \end{align*}
    provided $\frac{8\epsilon}{T}\leq 10\gamma$, which in particular holds if $T\geq\frac 45$.
    Therefore
    \begin{align}\label{eq:s(Lc)}
        \begin{split}
            \sigma\big((\Omega\backslash\Omega_t)\backslash\mathcal T(\Lambda)\big)
        &\leq \mu_0\left(t_0+\frac{4\epsilon}{T}\right)-\mu_0\left(10\gamma +t_0-\frac{4\epsilon}{T}\right)
        \\&\leq -\mu_0'\left(t_0+\frac{4\epsilon}{T}\right)\left(10\gamma+\frac{8\epsilon}{T}\right)
        \leq C\big(-\mu_0'(t_0)\big)\gamma.
        \end{split}
    \end{align}

    \noindent\emph{Step 4: Estimate for $ \sigma(\Omega^*\backslash\Omega_t)$.}
    Recall that   $\Omega^*=\{\zeta\in\partial\mathbb B: U_0(\zeta)>t_0\}$ so $\sigma(\Omega^*)=\mu_0(t_0)=\omega$.
    Using \eqref{eq:inclusions} we infer
    \begin{align}\label{eq:s(difflevelsets)}
        \sigma(\Omega^*\backslash\Omega_t)
        &\leq \mu_0(t_0)- \mu_0\left(t_0+\frac{4\epsilon}{T}\right)
        \leq\frac{4\epsilon}{T}
        \leq C\big(-\mu_0'(t_0)\big) \gamma,
    \end{align}
    where in the last step we have assumed $T\geq \frac 12$.

    \noindent\emph{Step 5: Conclusion.}  
    We notice that
    \begin{align*}
        \sigma(\Omega\backslash\Omega^*)
        &\leq \sigma (\Omega\backslash \Omega_t)+\sigma(\Omega_t\backslash\Omega^*).
    \end{align*}
    Using \eqref{eq:subsetsplit}, \eqref{eq:s(L)}, \eqref{eq:s(Lc)} and  \eqref{eq:s(difflevelsets)}  and recalling \eqref{eq:gamma} and that $t_0=(1-\omega^{1/d})^N$, one finally obtains  
    for $T\geq \frac 4 5$ that
\begin{align*}
    \sigma(\Omega\backslash\Omega^*)\leq C
    \left(\frac{C_{N,\Omega^*}(\zeta_1^N)}{\gamma_0}+\big(-\mu_0'(t_0)\big)\gamma_0\right)\left(1-\frac{C_{N,\Omega}(Q)}{C_{N,\Omega^*}(\zeta_1^N)}\right)^{1/2},
\end{align*}
    where $\gamma_0=\left({\binom{N+d}{d}\int_\omega^{\tilde\omega}\mu_0^{-1}(s)ds}\right)^{-1/2}>1$.
    Taking into account that $C_{N,\Omega^*}(\zeta_1^N)\leq 1$, 
    $-\mu_0'(t_0)\leq \frac{d}{N}(1-\omega^{1/d})^{1-N}$ and $(1-\omega^{1/d})^{1-N}\geq 1$, we conclude
    \begin{align*}
        \mathcal A_\sigma(\Omega)
        &\leq\frac{2\sigma(\Omega\backslash\Omega^*)}{\sigma(\Omega)}\leq 
        C\frac{(1-\omega^{1/d})^{1-N} \gamma_0}{\omega}
        \left(1-\frac{C_{N,\Omega}(Q)}{C_{N,\Omega^*}(\zeta_1^N)}\right)^{1/2},
    \end{align*}
    which provides the desired estimate.

    If $T< \frac{4}{5}$, that is, $\epsilon>\frac{1}{\sqrt{5}}$, then 
    \begin{align*}
        \mathcal A_\sigma(\Omega)\leq \frac{1}{\sigma(\Omega)}\leq  C\frac{\epsilon}{\omega}
        \leq \frac{C}{\omega N^{d/2}\left(\int_\omega^{\tilde\omega}\mu_0^{-1}(s)ds\right)^{1/2}}\left(1-\frac{C_{N,\Omega}(Q)}{C_{N,\Omega^*}(\zeta_1^N)}\right)^{1/2},
    \end{align*}
    and since $(1-\omega^{1/d})^{1-N}\geq 1$, the result follows.
    
\end{proof}

\begin{proof}[Proof of Theorem~\ref{thm:Wehrl}]

For any $t_0\in(0,1)$, let
\begin{align*}
    \Phi_{0}(t)=\begin{cases}
        0 & \mbox{ if } 0\leq t\leq t_0,\\
        \Phi(t)-\Phi(t_0)-\Phi'(t_0)(t-t_0) & \mbox{ if } t_0< t\leq 1.
    \end{cases}
\end{align*}
Here and below, $\Phi'$ must be understood as the left derivative of $\Phi$.
Not only $\Phi_{0}$ is a continuous, convex function, 
but so is $\Phi_1(t)=\Phi(t)-\Phi_{0}(t)$.

Let
\begin{align*}
    \mathcal S_j=\frac{S_{\Phi_j}(Q)-S_{\Phi_j}(\zeta_1^N)}{\binom{N+d}{N}}, \quad j\in\{0,1\}.
\end{align*}
By Theorem~\ref{thm:LS16}, $\mathcal S_1\geq 0$ and we will prove 
\begin{align}\label{eq:S0est}
    \mathcal S_0\geq \frac{C}{N^d}(1-T)
\end{align}
To do so, notice that 
\begin{align*}
    \mathcal S_0=\int_{0}^{1} \Phi_0'(t)\big(\mu_0(t)-\mu(t)\big)dt
    =\int_{t_0}^{1} \big(\Phi'(t)-\Phi'(t_0)\big)\big(\mu_0(t)-\mu(t)\big)dt
\end{align*}

Since $\Phi$ is non-linear, there are $0< a<b<1$ such that $\Phi'(a)<\Phi'(b)$.
Let us consider $\Phi_0$ with $t_0=a$, so
\begin{align*}
    \mathcal S_0
    &=\int_a^1 \big(\Phi'(t)-\Phi'(a)\big)\big(\mu_0(t)-\mu(t)\big)dt.
\end{align*}

Let $N_0$ be the smallest integer  such that  $t^{\tilde\omega}_0=(1-\tilde\omega^{1/d})^N<a$.
Therefore, by Lemma~\ref{lem:uniquet*}, if $t^*\in(a,1)$ and $\mu_0(t)<\mu(t)$ for $t\in(a, t^*)$.
Otherwise, $t^*\leq a$ and  $\mu_0(t)>\mu(t)$ in $(a,1)$.

In the latter case (which is always satisfied if $a\geq T^*$ according to Lemma~\ref{lem:uniquet*}), then 
\begin{align*}
    \mathcal S_0
    &=\int_a^1 \big(\Phi'(t)-\Phi'(a)\big)\big(\mu_0(t)-\mu(t)\big)dt
    \geq  \big(\Phi'(b)-\Phi'(a)\big)\int_{b}^1 \big(\mu_0(t)-\mu(t)\big)dt.
\end{align*}

If  $t^*>a$, we have 
\begin{align}\label{eq:S0t*}
    \mathcal S_0
    &=\int_a^1 \big(\Phi'(t)-\Phi'(t^*)\big)\big(\mu_0(t)-\mu(t)\big)dt
    +\big(\Phi'(t^*)-\Phi'(a)\big)\int_a^1 \big(\mu_0(t)-\mu(t)\big)dt.
\end{align}
Notice that both terms in the right hand side are non-negative.
Indeed, this follows from Lemma \eqref{lem:uniquet*} for the first term and from Proposition~\ref{prop:ineqLS} for the second one. 

On the one hand, assuming that $a<t^*\leq b$, we estimate \eqref{eq:S0t*} as follows 
\begin{align*}
    \mathcal S_0
    &\geq  \big(\Phi'(b)-\Phi'(t^*)\big)\int_b^1 \big(\mu_0(t)-\mu(t)\big)dt
    +\big(\Phi'(t^*)-\Phi'(a)\big)\int_a^1 \big(\mu_0(t)-\mu(t)\big)dt.
\end{align*}
On the other hand, if $a<b\leq t^*$,
we keep only the second term in \eqref{eq:S0t*}, so 
\begin{align*}
    \mathcal S_0
    &\geq \big(\Phi'(b)-\Phi'(a)\big)\int_a^1 \big(\mu_0(t)-\mu(t)\big)dt.
\end{align*}

Hence, it remains to prove the suitable lower bound for  $\int_a^1 \big(\mu_0(t)-\mu(t)\big)dt$ if $a\leq t^*$ and for $\int_b^1 \big(\mu_0(t)-\mu(t)\big)dt$ if $b\geq t^*$.
In the second case, this follows by
Lemma~\ref{lem:estimateintegral} with $t_0=b$ provided $T\geq T_1=\frac{1+T_0}{2}>b$. Otherwise, if $T\leq T_1$, then 
$$\int_{b}^1 \big(\mu_0(t)-\mu(t)\big)dt\geq \int_{T_1}^1 \mu_0(t)dt\geq \frac{(1-T_1)^{d+1}}{N^d}\geq \frac{C}{N^d}(1-T).$$
If $a\leq t^*$, since $a>t_0^{\tilde\omega}$ and by Lemma~\ref{lem:uniquet*} we know $\mu_0(t)<\mu(t)$ in $(t_0^{\tilde\omega}, a)$, we have
$$\int_{a}^1 \big(\mu_0(t)-\mu(t)\big)dt\geq \int_{t_0^{\tilde\omega}}^1 \big(\mu_0(t)-\mu(t)\big)dt\geq \frac{C}{N^d}(1-T).$$
where the last estimate follows from Lemma~\ref{lem:lower2}.

The final result  is then a consequence of  \eqref{eq:S0est}, Lemma~\ref{lem:valuemin}
and \eqref{eq:binompot}.
\end{proof}

\section{Further observations}\label{sec:further}
In this section we collect a series of observations or consequences of the main results. 
Namely, in the same way we did in \cite{GFOC25} for $d=1$, we recover the results in  the Bargmann-Fock space of \cites{GGRT, FNT25} by taking the limit $N\to\infty$, we prove the sharpness of our main results, and we generalize them to operators acting on $\mathbb P^d_N$. 
Since all these arguments mimic those in \cite{GFOC25},  we only stress the main differences that need to be taken into account in the case of higher dimensions.

\subsection{Quantitative estimates in the  Bargmann–Fock Space}\label{sec:Fock}

Given $d\in\mathbb N$, let $\mathcal F^2$ denote the Bargmann-Fock space of entire functions $f(z)$, $z\in\mathbb C^d$,  with
\begin{align*}
    \|f\|_{\mathcal F^2}^2=\int_{\mathbb C^d}|f(z)|^2 e^{-\pi |z|^2}dz<\infty.
\end{align*}
Endowed with the inner product 
\begin{align*}
    \langle f,g\rangle_{\mathcal F^2}=\int_{\mathbb C^d} f(z)\overline{g(z)}e^{-\pi |z|^2} dz,
\end{align*}
where $dz=\prod_{j=1}^ddz_j$ and $dz_j=dx_j dy_j$ for $z_j=x_j+iy_j\in\C$.  The Bargmann-Fock space
$\mathcal F^2$ is a reproducing kernel Hilbert space with kernel $k(z,\zeta)=e^{z\cdot\bar\zeta}$.

Quantitative estimates for the concentration inequality and for a generalized Wehrl entropy were obtained in\cite{GGRT} and \cite{FNT25}, respectively. Both results hold without any constrain in the size of the domains for the concentration or in the range of non-linearity of  $\Phi$. 
Considering the limit as $N\to\infty$ of Theorem~\ref{thm:concentration_affine} and~\ref{thm:Wehrl_affine}, these results can be inferred following analogous arguments to those in Section 4 of \cite{GFOC25}. 
We notice that  the observed constrains in the size of allowed subsets for the concentration inequality, or in $N$ for a given convex function on the Wehrl entropy setting cause no restriction when studying the limit. 

In order to study the limit as $N\to\infty$, given any polynomial $q(z)$ and any domain $\mathcal W\subset\mathcal C^d$, we  define the following rescaled objects:
\begin{align*}
    q^N(z)&=q\left(\sqrt{\frac N\pi}z\right),\qquad
    \mathcal W^N=\sqrt{\frac\pi N}\mathcal W.
\end{align*}
Notice that $m(\mathcal W^N)$ decreases with $N$, so if  $m(\mathcal W)$ is fixed, there exists $N_0\in\mathbb N$ such that $m(\mathcal W^N)\leq\tilde \omega$ for $N\geq N_0$, and therefore Theorem~\ref{thm:concentration_affine}  can be applied fo $N\geq N_0$. 
Since the whole argument can be easily generalized from Section 4 of \cite{GFOC25}, we skip here all the computations. 

Regarding the stability for the generalized Wehrl entropy, we need to observe that given any convex function $\Phi$, since the result in $\mathcal P_N^d$ holds for $N\geq N_\Phi$, we can take the limit without any problem in Theorem~\ref{thm:Wehrl_affine}.

\subsection{Sharpness}
\label{sec:optimality}

Theorems~\ref{thm:concentration} and~\ref{thm:Wehrl} are sharp in terms of the powers of the different notions of closeness to the optimal values, as seen for dimension one in Section 5 of \cite{GFOC25}. 
The generalization of this result to any dimension follows the same principles, working on the affine coordinates of $\mathbb {CP}^d$. 

On the one hand, in the case of the local concentration, considering the limit to the Fock-Bargmann space of Section~\ref{sec:Fock} and   \cite{GGRT}*{Corollary 7.3}, the sharpness is inferred by contradiction.

On the other hand, the counterexample  of \cite{GFOC25}*{Proposition 5.1} for $d=1$ allows to infer the sharpness in the Wehrl entropy inequality. 
Indeed, we notice that if $q\in\mathcal P^d_N$ with $d\geq 2$ and $q$ only depends on $z_1\in\mathbb C$, then 
\begin{align*}
    \|q\|_{\mathcal P^d_N}=\|q\|_{\mathcal P^1_N}.
\end{align*}
Therefore, for any $q(z)=q(z_1)\in\mathcal P^d_N$ with $\|q\|_{\mathcal P^d_N}=1$ and
making the dependence of $D_N$ on $d$ explicit, we have
\begin{align*}
    D_{d,N}(q)\leq  \min\big\{\|q-\kappa_N(\cdot, (w_1, 0,\dots, 0))\|_{\mathcal P^d_N}: w_1\in\mathbb C, \theta\in[0,2\pi]\big\}=D_{1,N}(q).
\end{align*} 
Similarly, $ S_{d,N, \Phi}(q)=S_{1,N, \Phi}(q)$. Hence $q(z)=1+\epsilon z_1$, up to renormalization, gives us the desired counterexample.

\subsection{Stability of general operators}
\label{sec:generalop}
Theorems~\ref{thm:concentration} and~\ref{thm:Wehrl} can be generalized to positive-semidefinite operators $\rho:\mathbb P^d_N\to\mathbb P^d_N$ with $\Tr(\rho)=1$ or, as seen in Section~\ref{subsec:irred}, to general states. 
Note that our previous results involve rank-one operators or pure states.

For any $\rho$ as above, let us consider the function
\begin{align}\label{eq:U}
    U(\zeta)=\langle K_N(\cdot,  \zeta), \rho (K_N(\cdot, \zeta))\rangle_{\mathbb P^d_N}.
\end{align}
This function agrees with the Husimi function $u_\rho$ introduced in \eqref{eq:Husimi} up to the identification of  $\zeta\in\partial\mathbb B$ with $\mathcal R\in SU(d+1)$, where $\zeta=\mathcal R\eta$ for a prefixed $\eta\in\partial\mathbb B$.
Notice that if $\rho=\langle\cdot,Q\rangle_{\mathbb P^d_N} Q$, then $U(\zeta)=|\langle Q, K_N(\cdot,\zeta) \rangle_{\mathbb P^d_N}|^2=|Q(\zeta)|^2$, as in \eqref{eq:Udef}.
The concentration of $\rho$ in $\Omega\subset\partial\mathbb B$  is defined as 
\begin{align*}
    C_{N, \Omega}(\rho)&=\frac{\int_\Omega U(\zeta)d\sigma(\zeta)}{\int_{\mathbb C^d} U(\zeta)d\sigma(\zeta)}.
\end{align*}
In addition, its distance to the projection to reproducing kernels (coherent states) is given by
\begin{align*}
    D_N(\rho)=\min\{\|\rho-\pi_{\eta}\|_1: \eta\in\partial\mathbb B\},
\end{align*}
where $\pi_{\eta}(Q)=\langle Q,K_N(\cdot, \eta)\rangle_{\mathbb P^d_N} K_N(\cdot, \eta)$ and $\|\pi\|_1=\Tr |\pi|$ for any operator $\pi$.

\begin{thm}\label{thm:concentration_generalop}
    For any $d\in\mathbb N$ there exist constants $C>0$ and $\tilde\omega\in(0,1]$ such that for any measurable set $\Omega\subset\partial\mathbb B$ with $\sigma(\Omega)\in(0,\tilde\omega)$ and any  positive-semidefinite  operator $\rho:\mathbb P_N^d\to \mathbb P_N^d$  with $\Tr \rho=1$, there holds
   \begin{align*}
       D_N(\rho)^2\leq \frac{C}{N^d\int_{\sigma(\Omega)}^{\tilde\omega}(1-s^{1/d})^N ds}
        \left(1-\frac{C_{N,\Omega}(\rho)}{C_{N,\Omega^*}(\pi_{\theta_1})}\right).
    \end{align*}
    where $\Omega^*$ is given in \eqref{eq:Omega*} and $\theta_1=(1, 0,\dots, 0)\in\partial\mathbb B$.
 Moreover, 
    \begin{align}\label{eq:conc_set_generalop}
        \mathcal A_\sigma(\Omega)^2\leq \frac{C}{\alpha\big(\sigma(\Omega)\big)}\left(1-\frac{C_{N,\Omega}(\rho)}{C_{N,\Omega^*}(\pi_{\theta_1})}\right), 
    \end{align}
with $\alpha(\omega)$ as in \eqref{eq:alpha_omega}.
\end{thm}

Finally, we recall the definition of the generalized Wehrl entropy $\mathcal S_{N,\Phi}$ in \eqref{eq:entropyop}, which is equivalent to 
\begin{align*}
S_{N,\Phi}(\rho)&=-\binom{N+d}{N}\int_{\mathbb C^d} \Phi\left(U(\zeta)\right) d\sigma(\zeta),
\end{align*}
with $U$ as in \eqref{eq:U}.
The following result is the  generalization of Theorem~\ref{thm:Wehrl} and a quantitative version of Theorem~\ref{thm:LS16op}.

\begin{thm}\label{thm:Wehrl_generalop}
  Let $\Phi:[0,1]\to\mathbb R$ be a  convex, non-linear function and let $d\in\mathbb N$.  Then there exist  constants $C > 0$  and $N_\Phi\in\mathbb N$ (depending only on $\Phi$ and $d$) such that for any $N\geq N_\Phi$, the following holds: Let $\rho:\mathbb P_N^d\to \mathbb P_N^d$ be a  positive-semidefinite  operator with $\Tr \rho=1$, then
  \begin{align*}
      D_N(\rho)^2\leq C\big(\mathcal S_{N,\Phi}(\rho)-\mathcal S_{N,\Phi}(\rho_0)\big),
  \end{align*}
  where $\rho_0$ is any coherent state, i.e. $\rho_0=\pi_\eta$ for any $\eta\in\partial\mathbb B$.
\end{thm}

The proof of both results work as those of Theorems~\ref{thm:concentration} and~\ref{thm:Wehrl}, respectively, with mild adaptations as those  pointed out in \cite{GFOC25}*{Section 5} in the one-dimensional case.  Therefore, we omit the details here and invite the interested reader to apply remarks in \cite{GFOC25}*{Section 5} to this setting, working on either $\mathbb P^d_N$ or $\mathcal P^d_N$. Indeed, Theorems~\ref{thm:concentration_generalop} and~\ref{thm:Wehrl_generalop} could have been written similarly for operators $\mathcal P_N^d\to\mathcal P_N^d$. 
In this setting, we could consider the limit as $N\to\infty$ and deduce analogous results in the Bargmann-Fock space, recovering in particular the stability of the generalized Wehrl conjecture of \cite{FNT25}*{Theorem~3}.

\subsection*{Acknowledgments}
{MAGF has been partially supported by the grants PID2024-156055NA-I00 and CEX2023-001347-S, funded by MICIU/AEI/10.13039/501100011033.
She also gratefully acknowledges the CRM-MdM Chair or Excellence, funded by grant CEX2020-001084-M.
JOC  has been supported by grants PID2024-160033NB-I00 and CEX2020-001084-M by the Agencia Estatal de Investigación and by  2024 ICREA 00142 grant by the Generalitat de Catalunya.  
We thank Alexei Kulikov and Joaqu\'in P\'erez for helpful conversations.}

\DefineSimpleKey{bib}{archiveprefix}{}

\BibSpec{arXiv}{
  +{}{\PrintAuthors}{author}
  +{,}{ \textit}{title}
  +{}{ \parenthesize}{date}
  +{,}{ arXiv }{eprint}
}


\begin{bibdiv}
\begin{biblist}
\bib{Barbosa}{article}{
author={Barbosa, J. Lucas},
author={do Carmo, Manfredo},
author={Eschenburg, Jost},
title={Stability of hypersurfaces of constant mean curvature in Riemannian manifolds},
journal = {Math. Z.},
ISSN = {0025-5874},
year={1988},
volume={197}, 
number = {1},
pages = {123--138},
DOI = {10.1007/BF01161634},
URL = {https://eudml.org/doc/183722},
}


\bib{Frank23}{article}{
 Author = {Frank, Rupert L.},
 Title = {Sharp inequalities for coherent states and their optimizers},
 Journal = {Adv. Nonlinear Stud.},
 ISSN = {1536-1365},
 Volume = {23},
 Pages = {28},
 Note = {Id/No 20220050},
 Year = {2023},
}

\bib{FNT25}{article}{
  author = {Frank, Rupert L.},
  author = {Nicola, Fabio},
  author = {Tilli, Paolo},
  title = {The generalized Wehrl entropy bound in quantitative form},
  journal = {J. Eur. Math. Soc.},
  year = {2025},
 }

 \bib{GFOC25}{article}{
   author={Garc\'ia-Ferrero, Mar\'ia \'Angeles},
   author={Ortega-Cerd\'a, Joaquim},
   title={Stability of the concentration inequality on polynomials},
   journal={Comm. Math. Phys.},
   volume={406},
   date={2025},
   number={5},
   pages={Paper No. 112, 35},
   issn={0010-3616},
}

\bib{GGRT}{article}{
  author={G\'omez, Jaime},
  author={Guerra, Andr\'e},
  author={Ramos, Jo\~ao P.G.},
  author={Tilli, Paolo},
  title={Stability of the Faber-Krahn inequality for the short-time Fourier transform},
  journal={Invent. Math.},
  volume={236},
  pages={779-836},
  date={2024},
}


\bib{Hardy}{book}{
  author = {Hardy, G. H.},
  author = {Littlewood, John E.},
  author = {P\'olya, George},
  isbn = {0521358809},
  publisher = {Cambridge University Press},
  title = {Inequalities},
  year = {1988},
}


\bib{KNOCT}{article}{
   author={Kulikov, Aleksei},
   author={Nicola, Fabio},
   author={Ortega-Cerdà, Joaquim},
   author={Tilli, Paolo},
   title={A monotonicity theorem for subharmonic functions on manifolds},
   journal={Adv. Math.},
   volume={479},
   date={2025},
   pages={Paper No. 110423, 18},
   issn={0001-8708},
}

\bib{LiebSolovej14}{article}{
  author={Lieb, Elliott H.},
  author={Solovej,Jan Philip},
  title={Proof of an entropy conjecture for Bloch coherent spin states and its generalizations},
   journal={Acta Math},
   volume={212},
   date={2014},
   number={2},
   pages={379--398},
}

\bib{LiebSolovej16}{article}{
 author = {Lieb, Elliott H.},
 author = {Solovej, Jan Philip},
 title = {Proof of the {Wehrl}-type entropy conjecture for symmetric {{\({SU(N)}\)}} coherent states},
 journal = {Commun. Math. Phys.},
 issn = {0010-3616},
 volume = {348},
 number = {2},
 pages = {567--578},
 year = {2016},
}

\bib{Nardulli09}{article}{
   author={Nardulli, Stefano},
   title={The isoperimetric profile of a smooth Riemannian manifold for
   small volumes},
   journal={Ann. Global Anal. Geom.},
   volume={36},
   date={2009},
   number={2},
   pages={111--131},
   issn={0232-704X},
}


\bib{NRT25}{arXiv}{
 author = {Nicola, Fabio},
 author = {Riccardi, Federico},
 author = {Tilli, Paolo},
 title = {The {Wehrl}-type entropy conjecture for symmetric ${SU}({N})$ coherent states: cases of equality and stability},
 year = {2025},
 url = {https://arxiv.org/abs/2412.10940},
 eprint = {2412.10940},
 archiveprefix={arXiv},
}


\bib{NRT25b}{arXiv}{
 author = {Nicola, Fabio},
 author = {Riccardi, Federico},
 author = {Tilli, Paolo},
 title = {An elementary approach to {Wehrl}-type entropy bounds in quantitative form},
 year = {2025},
 url = {https://arxiv.org/abs/2512.04245},
 eprint = {2512.04245}, 
 archiveprefix={arXiv},
}

\bib{Oss}{article}{
   author={Osserman, Robert},
   title={The isoperimetric inequality},
   journal={Bull. Amer. Math. Soc.},
   volume={84},
   date={1978},
   number={6},
   pages={1182--1238},
   issn={0002-9904},
}

\bib{Rudin}{book}{
 author = {Rudin, Walter},
 title = {Function theory in the unit ball of {{\({\mathbb{C}}^ n\)}}},
 edition = {Reprint of the 1980 original},
 series = {Class. Math.},
 issn = {1431-0821},
 isbn = {978-3-540-68272-1},
 year = {2008},
 publisher = {Berlin: Springer},
 language = {English},
}

\bib{Shiffman15}{article}{
   author={Shiffman, Bernard},
   title={Uniformly bounded orthonormal sections of positive line bundles on
   complex manifolds},
   conference={
      title={Analysis, complex geometry, and mathematical physics: in honor
      of Duong H. Phong},
   },
   book={
      series={Contemp. Math.},
      volume={644},
      publisher={Amer. Math. Soc., Providence, RI},
   },
   isbn={978-1-4704-1464-1},
   date={2015},
   pages={227--240},
}

\end{biblist}
\end{bibdiv}
\end{document}